\documentclass[12pt,twoside]{amsart}
\usepackage{amssymb,amsmath,amsfonts,amsthm}
\usepackage{mathrsfs}
\usepackage[X2,T1]{fontenc}
\usepackage[applemac]{inputenc}
\usepackage{cite}
\usepackage{latexsym}
\usepackage{comment}
\usepackage{colortbl}
\def\Im{\mathop{\rm Im}\nolimits}
\def\Re{\mathop{\rm Re}\nolimits}

\def\Im{\mathop{\rm Im}\nolimits}
\def\Re{\mathop{\rm Re}\nolimits}

\def\R{\mathbb R}
\def\C{\mathbb C}
\def\N{\mathbb N}

\def\ds{\displaystyle}
\def\SG{\textbf{\textrm{SG}}}

\newcommand\dslash{d\llap {\raisebox{.9ex}{$\scriptstyle-\!$}}}
\usepackage{colortbl}

\newcommand{\beqsn}{\arraycolsep1.5pt\begin{eqnarray*}}
\newcommand{\eeqsn}{\end{eqnarray*}\arraycolsep5pt}
\newcommand{\beqs}{\arraycolsep1.5pt\begin{eqnarray}}
\newcommand{\eeqs}{\end{eqnarray}\arraycolsep5pt}

\newtheorem{theorem}{Theorem}
\newtheorem{lemma}{Lemma}

\newtheorem{proposition}{Proposition}
\newtheorem{definition}{Definition}

\newtheorem{remark}{Remark}

\catcode`\@=11

\renewcommand{\section}%
   {\setcounter{equation}{0}\@startsection {section}{1}{\z@}{-3.5ex plus -1ex
  minus -.2ex}{2.3ex plus .2ex}{\Large\bf}}

\title[$3$-evolution equations in Gelfand-Shilov spaces]{The Cauchy problem for $3$-evolution equations with data in Gelfand-Shilov spaces}

\author[A. Arias Junior]{Alexandre Arias Junior}
\address{ Department of Mathematics, Federal University of Paran\'a, CEP 81531-980, Curitiba, Brazil}
\email{arias@ufpr.br}

\author[A. Ascanelli]{Alessia Ascanelli}
\address{Dipartimento di Matematica ed Informatica\\Universit\`a di Ferrara\\
Via Machiavelli 30\\
44121 Ferrara\\
Italy}
\email{alessia.ascanelli@unife.it}
\author[M. Cappiello]{Marco Cappiello}
\address{Dipartimento di Matematica ``G. Peano'' \\Universit\`a di Torino\\
Via Carlo Alberto 10\\
10123 Torino\\
Italy}
\email{marco.cappiello@unito.it}

\thanks{The first author would like to thank Funda\c{c}\~ao Arauc\'aria for the financial support during the development of this paper. The second author has been supported by the Italian national research funds FFABR2017 and GNAMPA 2020. The third author is supported by the fund GNAMPA 2020.}

\begin{document}


\begin{abstract}
We consider the Cauchy problem for a third order evolution operator $P$ with $(t,x)$-depending coefficients and complex valued lower order terms. We assume the initial data to be Gevrey regular with an exponential decay at infinity, that is, the data belong to some Gelfand-Shilov space of type $\mathscr{S}$. Under suitable assumptions on the decay at infinity of the imaginary parts of the coefficients of $P$  
we prove the existence of a solution with the same Gevrey regularity of the data and we describe its behavior for $|x| \to\infty$. 
\end{abstract}

\maketitle



\section{Introduction and main result}\label{section_introduction}
Let us consider for $(t,x) \in [0,T] \times \R$ the Cauchy problem in the unknown $u=u(t,x)$:
\begin{equation}\label{Cauchy_problem_in_introduction}
\begin{cases}
P(t, x, D_t, D_x) u(t,x)= f(t, x)	\\
u(0,x) = g(x)
\end{cases},
\end{equation}
where
\begin{equation}\label{diffP}
P(t, x, D_t, D_x) = D_t + a_p(t)D^{p}_{x} + \displaystyle\sum_{j=0}^{p-1}a_j(t,x)D^{j}_{x},
\end{equation}
with $D=\frac1i\partial,$ $p\geq 2$, $a_p\in C([0,T],\R), a_p(t) \neq 0$ for $t \in [0,T]$, and $a_j \in C([0,T],C^\infty(\R;\C))$, $j = 0, \ldots, p-1$. The operator $P$ is known in literature as $p-$evolution operator, cf. \cite{Mizo}, and $p$ is the evolution degree. The well posedness of \eqref{Cauchy_problem_in_introduction}, \eqref{diffP} has been investigated in various functional settings for arbitrary $p$, cf.\cite{ascanelli_chiara_zhanghirati_well_posedness_of_cauchy_problem_for_p_evo_equations, ascanelli_chiara_zanghirati_necessary_condition_for_H_infty_well_posedness_of_p_evo_equations, ascanelli_cappiello_wheigted_energy_estimates_p_evolution_equations_SG_classes, cicognani_colombini_cauchy_problem_for_p_evolution_equations}. Further results concern special values of $p$ which correspond to classes of operators of particular interest in Mathematical Physics, cf. \cite{ascanelli_cappiello_schrodinger_equations_Gelfand_shillov, ACR, baba_H_infty_well_posedness_schrodinger_equations, CRJEECT, ichinose_remarks_cauchy_problem_schrodinger_necessary_condition, kajitani_baba_cauchy_problem_schrodinger_equations} for the case $p=2$ and \cite{AAC3evolGevrey} in the case $p=3$. 
The condition that $a_p$ is real valued means that the principal symbol of $P$ (in the sense of Petrowski)  has the real characteristic $\tau=-a_p(t)\xi^p$;
this guarantees that the operator \eqref{diffP} satisfies the assumptions of Lax-Mizohata theorem.
The presence of complex valued coefficients in the lower order terms of \eqref{diffP} plays a crucial role in the analysis of the problem \eqref{Cauchy_problem_in_introduction} in all the above mentioned papers. In fact, 
when the coefficients $a_j(t,x)$, $j = 0, \ldots, p-1,$ are real valued and of class $\mathcal B^\infty$ with respect to $x$ (uniformly bounded together with all their $x$-derivatives),  it is well known that the problem (\ref{Cauchy_problem_in_introduction}) is well-posed in $L^2(\R)$ (and in $L^2$-based Sobolev spaces $H^{m}$, $m\in\R$). On the contrary if any of the coefficients $a_j(t,x)$ are complex valued, then in order to obtain well-posedness either in $L^2(\R)$, or in $H^{\infty}(\R)=\cap_{m\in\R}H^m(\R)$, some decay conditions at infinity on the imaginary part of the coefficients $a_j$ are needed
(see \cite{ascanelli_chiara_zanghirati_necessary_condition_for_H_infty_well_posedness_of_p_evo_equations}, \cite{ichinose_remarks_cauchy_problem_schrodinger_necessary_condition}). 

Sufficient conditions for well-posedness in $L^2$ and $H^\infty$ have been given in \cite{baba_H_infty_well_posedness_schrodinger_equations} and\cite{kajitani_baba_cauchy_problem_schrodinger_equations} for the case $p=2$, in \cite{ascanelli_chiara_zhanghirati_well_posedness_of_cauchy_problem_for_p_evo_equations} for larger $p$.
Considering the Cauchy problem \eqref{Cauchy_problem_in_introduction} in the framework of weighted Sobolev-Kato spaces $H^m= H^{(m_1,m_2)}$, with $m=(m_1,m_2) \in \R^2$, defined as
\begin{equation} \label{KatoSobolev}
	H^{m}(\R)=\{u \in \mathscr{S}'(\R) : \langle x \rangle^{m_2} \langle D_{x} \rangle^{m_1} u \in L^{2}(\R) \},
	\end{equation}
where the analysis of regularity and decay/growth of the solution can be performed simulta\-ne\-ous\-ly, and assuming the coefficients $a_j$ to be polynomially bounded, the second and the third author obtained in \cite{ascanelli_cappiello_wheigted_energy_estimates_p_evolution_equations_SG_classes} well-posedness also in the Schwartz space $\mathscr{S}(\R)$ of smooth and rapidly decreasing functions and in the dual space $\mathscr{S}'(\R)$. We recall that $\mathscr{S}(\R) = \cap_{m \in \R^2} H^{m}(\R)$ and $\mathscr{S}'(\R) = \cup_{m \in \R^2} H^{m}(\R)$.
 In short, the above mentioned results can be summarized as follows: if 
\begin{equation}\label{genericdecay}
|\partial^{\beta}_{x} \Im a_{j}(t,x)| \leq C_{\beta} \langle x \rangle^{-\frac{j}{p-1}\sigma - |\beta|}, \quad (t,x) \in [0,T]\times \R,\, \beta \in \N_{0},\ C_\beta>0, j=0,\ldots, p-1,
\end{equation}
the problem \eqref{Cauchy_problem_in_introduction} is well-posed in:
\\
- $L^{2}(\R)$, $H^{m}(\R)$ for every $m \in\R^2$ when $\sigma > 1$;
\\
- $H^{\infty}(\R)$, $\mathscr{S}(\R)$  
when $\sigma = 1$.
In general, a finite loss of regularity of the solution with respect to the initial data is observed in the case $\sigma=1$.

\medskip
\noindent
Now we want to consider the case when an estimate of the form \eqref{genericdecay} for $j=p-1$ holds for some $\sigma \in (0,1).$
In this situation, there are no results in the literature for $p$-evolution operators of arbitrary order. In \cite{CRJEECT, kajitani_baba_cauchy_problem_schrodinger_equations} the case $p=2$, which corresponds to Schr\"odinger-type equations, is considered assuming $0<\sigma<1$ and $C_\beta = C^{|\beta|+1} \beta!^{s_0}$ for some $s_0 \in (1, 1/(1-\sigma))$ in \eqref{genericdecay}. The authors find well posedness results in certain Gevrey spaces of order $\theta$ with $s_0 \leq \theta<1/(1-\sigma)$, namely in the class
 $$\mathcal{H}^\infty_\theta= \bigcup_{\rho >0}H^m_{\rho;\theta},\quad  H^m_{\rho;\theta} := \{u \in L^2\vert\  e^{\rho \langle D \rangle^{1/\theta}}u \in H^m \},\, m\in\R.$$
In both papers, starting from data $f,g$ in $H^m_{\rho;\theta}$ for some $\rho>0$ the authors obtain a solution in $H^m_{\rho-\delta;\theta}$ for some $\delta >0$ such that $\rho-\delta >0.$ This means a sort of loss of regularity in the constant $\rho$ which rules the Gevrey behavior. We also notice that the condition $s_0 \leq \theta <1/(1-\sigma)$ means that the rate of decay of the coefficients of $P$ imposes a restriction on the spaces $\mathcal{H}^\infty_\theta$ in which the problem \eqref{Cauchy_problem_in_introduction} is well posed. Finally, the case $\theta >s_0=1/(1-\sigma)$ is investigated in \cite{ACR} where the authors prove that a decay condition as $|x|\to\infty$  on a datum in $H^m$, $m\geq 0$, produces a solution with (at least locally) the same regularity as the data, but with a different behavior at infinity.
In the recent paper \cite{ascanelli_cappiello_schrodinger_equations_Gelfand_shillov} the role of data with exponential decay on the regularity of the solution has been also analyzed for 2-evolution equations in arbitrary space dimension, in the frame of Gelfand-Shilov type spaces which can be seen as the global counterpart of classical Gevrey spaces, cf. Subsection \ref{section_gelfand_shilov_spaces}. In particular, it is proved that starting from data with an exponential decay at infinity, we can find a solution with the same Gevrey regularity of the data but with a possible exponential growth at infinity in $x$. Moreover, this holds for every $\theta >s_0$.  Finally, the result in \cite{ascanelli_cappiello_schrodinger_equations_Gelfand_shillov} is proved under the more general assumption with respect to \eqref{genericdecay} 
that the coefficients $a_1,\ a_0$ may admit an algebraic growth at infinity, namely
\begin{eqnarray} \label{alggrowth}
|\partial^{\beta}_{x} \Im a_1(t,x)| &\leq& C^{|\beta|+1} \beta!^{s_0} \langle x \rangle^{-\sigma - |\beta|},
\\ \label{alggrowth2}
 |\partial^{\beta}_{x} \Re a_1(t,x)| +|\partial^{\beta}_{x} a_0(t,x)| &\leq& C^{|\beta|+1} \beta!^{s_0} \langle x \rangle^{1-\sigma - |\beta|}.
 \end{eqnarray}

Recently, we started to consider the case $p=3$ in a Gevrey setting in one space dimension under assumption \eqref{genericdecay} with $j=p-1$ taking $\sigma \in (0,1)$. This case is of particular interest because linear $3$-evolution equations can be regarded as linearizations of relevant physical semilinear models like KdV and KdV-Burgers equation and their generalizations, see for instance \cite{JM, Kato, KdV, KKN,TMI}. There are some results  concerning KdV-type equations  with coefficients not depending on $(t,x)$ in the Gevrey setting, see \cite{GHHP,Goubet,hhp}. Our aim is to treat the more general case of variable coefficients. The present paper and \cite{AAC3evolGevrey} are devoted to establish the linear theory which is a preliminary step to treat the semilinear case. In a future paper we shall consider the case when the coefficients $a_j$ may depend also on $u$ following the approach developed in \cite{ABtame} in the $H^\infty$ setting.
\\
\indent
Also for the case $p=3$, assuming a condition of the form \eqref{genericdecay} with $j=2$ on the term $a_2$ for some $\sigma \in (0,1)$, namely \beqs\label{c}|\partial_x^\beta\Im a_2(t,x)|\leq C_\beta\langle x\rangle^{-\sigma-\beta},\eeqs is enough to lose in general well posedness both in $H^\infty$ and in $\mathscr{S}$, since the necessary condition for $H^\infty$ well-posedness 
\begin{equation}\label{neccond}\sup_{x\in\R}\min_{0\leq\tau\leq t\leq T}\int_{-\varrho}^\varrho 
\Im a_{2}(t,x+3 a_3(\tau)\theta)d\theta\leq M\log(1+\varrho)+N,\qquad 
\forall \varrho>0,\; {\rm for\ some\ } M,N>0 \end{equation}
proved in \cite{ascanelli_chiara_zanghirati_necessary_condition_for_H_infty_well_posedness_of_p_evo_equations} is no more satisfied. Namely, well posedness in $H^\infty$ or in $\mathscr{S}$ may fail due to an infinite loss of regularity or of decay. To give an idea of the latter phenomenon, consider the following initial value problem
	\begin{equation}\label{CPexample} \begin{cases} D_t u+D^3_x u +a_2(t,x)D_x^2 u +a_1(t,x)D_x u +a_0(t,x)u =0 \\ u(0,x) = e^{-\langle x \rangle^{1-\sigma}} \end{cases}, \qquad (t,x) \in [0,T]\times \R, \end{equation}
	where $$a_2(t,x)= i(t-1)(1-\sigma) x \langle x \rangle^{-\sigma-1},$$
	$$a_1(t,x)= 2(t-1)(1-\sigma) [\langle x \rangle^{-\sigma-1}-(\sigma+1)x^2 \langle x \rangle^{-\sigma-3}],$$
	$$a_0(t,x)= i \langle x \rangle^{1-\sigma}+i(t-1)(1-\sigma^2)[3x \langle x \rangle^{-\sigma-3}-(\sigma+3)x^3 \langle x \rangle^{-\sigma-5}].$$
	Notice that the coefficients $a_j$ are analytic and satisfy conditions \eqref{genericdecay} for $j=2$ and  \eqref{alggrowth}, \eqref{alggrowth2} for $j=0,1$. Moreover the initial datum belongs to $\mathscr{S}(\R)$ since $\sigma \in (0,1)$. It is easy to verify that the problem \eqref{CPexample} admits the solution
	$$u(t,x)= e^{(t-1)\langle x \rangle^{1-\sigma}} \notin C([0,T], \mathscr{S}(\R)),$$
	if $T \geq 1$. Analogously, $u \notin  C([0,T], H^\infty(\R))$. More precisely, we notice that the solution has the same regularity as the initial data but it grows exponentially for $|x| \to \infty$ when $t \geq 1.$ This motivates us to study the effect of an exponential decay of the data on the solution of \eqref{Cauchy_problem_in_introduction}.

In the recent paper \cite{AAC3evolGevrey} we proved a result of well posedness in the space $\mathcal{H}^\infty_{\theta}$ for the problem \eqref{Cauchy_problem_in_introduction} which extends to the case $p=3$ the results obtained in \cite{CRJEECT, kajitani_baba_cauchy_problem_schrodinger_equations} for the case $p=2$ (at least in one space dimension). As in the latter case, also in \cite{AAC3evolGevrey} a loss of regularity in the index $\rho$ appears. However, the previous example suggests that this loss can be avoided assuming the initial data to admit a suitable exponential decay. The price to pay is a considerable loss of decay which may produce solutions admitting an exponential growth.  
In view of the considerations above, it is quite natural to analyze the problem \eqref{Cauchy_problem_in_introduction} when the initial data belong to Gelfand-Shilov spaces, cf. Subsection 2.1 for the definition.

In order to state our main result we need to recall the definition of Gevrey type $\SG$-symbol classes and of Gelfand-Shilov Sobolev spaces.

Given $\mu, \nu \geq 1$, $m = (m_1, m_2) \in \R^{2}$, we denote by $\SG^{m_1,m_2}_{\mu, \nu}(\R^{2})$ (or by $\SG^{m}_{\mu, \nu}(\R^{2})$) the space of all functions $p \in C^\infty(\R^{2})$ for which there exist $C,C_1 > 0$ such that 
	$$
	|\partial^{\alpha}_{\xi} \partial^{\beta}_{x}p(x,\xi)| \leq C_1 C^{|\alpha+\beta|}\alpha!^{\mu}\beta!^{\nu} \langle \xi \rangle^{m_1 -|\alpha|} \langle x \rangle^{m_2 - |\beta|}, \quad x,\xi \in \R, \alpha, \beta \in \N,$$
see also Definition \ref{definition_SG_symbol_classe_finite_order}. In the case $\mu=\nu$ we write  $\SG^{m_1,m_2}_{\mu}(\R^{2})$ instead of $\SG^{m_1,m_2}_{\mu, \mu}(\R^{2})$. In the following we shall obtain our results via energy estimates, hence we need to introduce the Gelfand-Shilov-Sobolev spaces
$H^{m}_{\rho;s,\theta}(\R)$ defined, for $m = (m_1, m_2), \rho = (\rho_1, \rho_2)$ in $\R^{2}$ and $\theta, s > 1$, by
$$
H^{m}_{\rho; s,\theta} (\R) = \{ u \in \mathscr{S}'(\R) : \langle x \rangle^{m_2} \langle D \rangle^{m_1} 
e^{\rho_2 \langle x \rangle^{\frac{1}{s}}} e^{\rho_1 \langle D \rangle^{\frac{1}{\theta}}} u \in L^{2}(\R) \},
$$
where $e^{\rho_1 \langle D \rangle^{\frac{1}{\theta}}}$ is the Fourier multiplier with symbol $e^{\rho_1 \langle \xi \rangle^{\frac{1}{\theta}}}$. When $\rho = (0, 0)$ we recover the usual notion of weighted Sobolev spaces \eqref{KatoSobolev}.

\medskip
Our pseudodifferential approach allows to consider more general $3$-evolution operators of the form
\begin{equation}\label{equation_main_class_of_3_evolution_of_pseudo_diff_operators}
P(t, D_t, x, D_x) = D_t + a_3(t, D_x) + a_2(t, x, D_x) + a_1(t, x, D_x) + a_0(t, x, D_x),
\end{equation} 
$t \in [0, T],\ x \in \R,$ where $a_3(t,D_x)$ is a pseudodifferential operator with symbol $a_3(t,\xi)\in\R$ while, for $j=0,1,2$, $a_j(t,x,D_x)$ are pseudodifferential operators with symbols $a_j(t,x,\xi)\in\C$. 
Notice that \eqref{diffP} in the case $p=3$ is a particular case of \eqref{equation_main_class_of_3_evolution_of_pseudo_diff_operators}.
Our main result reads as follows.
\begin{theorem}\label{theorem_main_theorem}
	Let $P(t,  x, D_t, D_x)$ be an operator as in ($\ref{equation_main_class_of_3_evolution_of_pseudo_diff_operators}$) and assume the following:
\begin{itemize}
	\item[(i)] $a_3 \in C([0, T], \SG^{3, 0}_{1}(\R^2))$, $a_3$ is real valued and there are $C_{a_3}, R_{a_3} > 0$ and $\sigma \in (0,1)$  such that
	$$
	|\partial_{\xi} a_3(t, \xi)| \geq C_{a_3}|\xi|^{2}, \qquad \forall |\xi| \geq R_{a_3}, \,\, \forall t \in [0, T];
	$$
	\item[(ii)]  $\Re  \, a_2 \in C([0, T], \SG^{2, 0}_{1, s_0}(\R^2))$, $\Im \, a_2 \in C( [0, T], \SG^{2, -\sigma}_{1, s_0}(\R^2))$;
	\item[(iii)]  $\Re  \, a_1 \in C([0, T], \SG^{1, 1-\sigma}_{1, s_0}(\R^2))$, $\Im \, a_1 \in C( [0, T], \SG^{1, -\frac{\sigma}{2}}_{1, s_0}(\R^2))$;
	\item[(iv)]  $a_0 \in C([0, T], \SG^{0, 1-\sigma}_{1, s_0}(\R^2))$.
\end{itemize}
Let $s, \theta > 1$ such that $s_0 \leq s < \frac{1}{1-\sigma}$ and $\theta > s_0$. Let moreover $f\in C([0,T]; H^{m}_{\rho; s, \theta}(\R))$ and $g\in H^{m}_{\rho; s, \theta}(\R)$, where $m=(m_1,m_2),\rho=(\rho_1,\rho_2) \in \R^2$ and $\rho_2 > 0$. Then the Cauchy problem \eqref{Cauchy_problem_in_introduction}
	admits a solution $u\in C([0,T]; H^{m}_{(\rho_1,-\tilde{\delta}); s,\theta}(\R))$ for every $\tilde{\delta}>0$, which satisfies the following energy estimate
	\beqs\label{energy_estimate}
	\| u(t) \|^{2}_{H^{m}_{(\rho_1, -\tilde{\delta}); s, \theta}} \leq C \left( \| g \|^{2}_{H^m_{\rho; s, \theta}} + \int_{0}^{t} \| f (\tau) \|^{2}_{H^{m}_{\rho; s, \theta}} d\tau \right),
	\eeqs
	for all $t\in [0,T]$ and for some $C>0$.
\end{theorem}

\begin{remark}
We notice that the solution obtained in Theorem \ref{theorem_main_theorem} has the same Gevrey regularity as the initial data but it may lose the decay exhibited at $t=0$ and admit an exponential growth for $|x| \to \infty$ when $t>0$. Moreover, the
loss $\rho_2+\tilde\delta$ for an arbitrary $\tilde\delta>0$ in the behavior at infinity is independent of $\theta$, $s$ and $\rho_1.$
Both these phenomena had been already observed in the case $p=2$, see \cite{ascanelli_cappiello_schrodinger_equations_Gelfand_shillov}.
\end{remark}

\begin{remark}
		Let us compare Theorem \ref{theorem_main_theorem} with the recent result obtained in \cite{AAC3evolGevrey}. In the latter paper, taking $a_0$ uniformly bounded, $a_1 \sim \langle x\rangle^{-\sigma/2} $ and $ a_2 \sim \langle x\rangle^{-\sigma}$ for some $\sigma\in (1/2,1)$, and the Cauchy data $f(t),g\in H^{(m_1,0)}_{(\rho_1,0);s,\theta}$ with $ s_0<\theta<1/(2(1-\sigma))$ we prove the existence of a unique solution $u \in C([0,T], H^{(m_1,0)}_{(\rho_1',0);s,\theta}(\R))$, for some $\rho_1' \in (0,\rho_1)$ i.e. a solution less regular than the data. Theorem \ref{theorem_main_theorem} in the present paper shows that if the data $f(t),g\in H^{(m_1,m_2)}_{(\rho_1,\rho_2);s,\theta}$ with $\rho_2 >0, \theta >s_0$ and $s_0 \leq s< 1/(1-\sigma)$, then there exists a solution $u \in C([0,T], H^{(m,0)}_{(\rho_1,-\tilde{\delta});s,\theta})$, $\forall \tilde{\delta}>0$, i.e. a solution with the same index $\rho_1$ as the data, but with a possible worse behavior at infinity: in particular, this solution may grow exponentially for $|x|\to\infty$.
Concerning the assumptions, with respect to the existing literature, in particular \cite{AAC3evolGevrey, kajitani_baba_cauchy_problem_schrodinger_equations}, in our result we allow:\\
- a polynomial growth of exponent $1-\sigma\in (0,1)$ for the coefficients $\Re a_1$ and $a_0$; \\
- an arbitrary Gevrey regularity index $\theta >s_0$ both for the data and for the solution, without any upper bound: namely there is no relation between the rate of decay of the data and the Gevrey regularity of the solution.
\end{remark}

\begin{remark} Part of the recent literature on $p$-evolution equations is focused on the research of necessary conditions for the well posedness of the problem \eqref{Cauchy_problem_in_introduction} in various functional settings, see \cite{dreher,  ascanelli_chiara_zanghirati_necessary_condition_for_H_infty_well_posedness_of_p_evo_equations, ichinose_remarks_cauchy_problem_schrodinger_necessary_condition}. Necessary conditions are usually expressed in an integral form as in \eqref{neccond} instead than via pointwise decay estimates as in \eqref{c}. As far as we know the only result of this type in the Gevrey setting concerns the case $p=2$, see \cite{dreher}. Our purpose is to investigate this problem in the next future for generic $p$.
\end{remark}

In order to help the reading of the next sections we briefly outline the strategy of the proof of Theorem \ref{theorem_main_theorem}.
Let
$$
iP = \partial_{t} + ia_3(t,D) + \sum_{j = 0}^{2} ia_j(t,x,D) = \partial_{t} + ia_{3}(t,D) + A(t,x,D).
$$
Noticing that $a_3(t,\xi)$ is real valued, we have 
\begin{align*}
\frac{d}{dt} \| u(t) \|^{2}_{L^{2}} &= 2Re\, (\partial_{t} u(t), u(t))_{L^{2}} \\ 
&= 2Re\, (iPu(t), u(t))_{L^2} - 2Re\, (ia_3(t,D) u(t), u(t))_{L^2} - 2Re\,(Au(t), u(t))_{L^2} \\
&\leq \| Pu(t) \|^{2}_{L^{2}} + \| u(t) \|^{2}_{L^{2}} - \,((A+A^{*})u(t), u(t))_{L^2}.
\end{align*}
Since $(A+A^{*})(t) \in \SG^{2, 1-\sigma}(\R^{2})$ we cannot derive an energy inequality in $L^2$ from the estimate above. The idea is then to conjugate the operator $iP$ by a suitable pseudodifferential operator $e^{\Lambda}(t,x,D)$ in order to obtain 
$$
(iP)_{\Lambda} := e^{\Lambda} (iP) \{e^{\Lambda}\}^{-1} = \partial_{t} + ia_3(t,D) + A_{\Lambda}(t,x,D),
$$
where $A_{\Lambda}$ still has symbol $A_\Lambda(t,x,\xi) \in \SG^{2, 1-\sigma}(\R^2)$ but with $Re\, A_{\Lambda} \geq 0$. In this way, with the aid of Fefferman-Phong (see \cite{fefferman_Phong_inequality}) and sharp G{\aa}rding (see Theorem $1.7.15$ of \cite{nicola_rodino_global_pseudo_diffferential_calculus_on_euclidean_spaces}) inequalities,  we obtain the estimate from below  
$$
Re\, (A_{\Lambda} v(t), v(t))_{L^{2}} \geq -c \| v(t) \|^{2}_{L^{2}},
$$ 
and therefore for the solution $v$ of the Cauchy problem associated to the operator $P_\Lambda$ we get
$$
\frac{d}{dt} \| v(t) \|^{2}_{L^{2}} \leq 
C(\| (iP)_{\Lambda}v(t) \|^{2}_{L^{2}} + \| v(t) \|^{2}_{L^{2}}).
$$
Gronwall inequality then gives  the desired energy estimate for the conjugated operator $(iP)_{\Lambda}$. By standard arguments in the energy method we then obtain that the Cauchy problem associated with $P_{\Lambda}$
\begin{equation}\label{equation_conjugated_cauchy_problem}
	\begin{cases}
	P_{\Lambda}(t,D_t, x, D_x) v(t,x) = e^{\Lambda}(t,x,D_x) f(t,x), \quad (t,x) \in [0,T] \times \R \\
	v(0,x)= e^{\Lambda}(0,x, D_x)g(x), \quad x \in \R
\end{cases},
\end{equation}
 is well-posed in the weighted Sobolev spaces $H^{m}(\R)$ in \eqref{KatoSobolev}.
Finally, we derive the existence of a solution of (\ref{Cauchy_problem_in_introduction}) from the well posedness of \eqref{equation_conjugated_cauchy_problem}.
In fact, if $u$ solves (\ref{Cauchy_problem_in_introduction}) then $v=e^{\Lambda}u$ solves (\ref{equation_conjugated_cauchy_problem}), and if $v$ solves (\ref{equation_conjugated_cauchy_problem}) then $u=\{e^{\Lambda}\}^{-1}v$ solves (\ref{Cauchy_problem_in_introduction}). In this step the continuous mapping properties of $e^{\Lambda}$ and $\{e^{\Lambda}\}^{-1}$ will play an important role.

The construction of the operator $e^\Lambda$ will be the core of the proof. The function $\Lambda(t,x,\xi)$ will be of the form
$$
\Lambda(t, x, \xi) = k(t) \langle x \rangle^{1-\sigma}_{h} + \lambda_2(x,\xi) + \lambda_1(x,\xi), \quad t \in [0,T],\, x, \xi \in \R,
$$
where $\lambda_1, \lambda_2 \in \SG^{0, 1-\sigma}_{\mu} (\R^2)$, $k \in C^1([0,T]; \R)$ is a non increasing function to be chosen later on and $\langle x \rangle_{h}=\sqrt{h^2+x^2}$ for some $h \geq 1$ to be chosen later on. The role of the terms $\lambda_1, \lambda_2, k$ will be the following:
\\
-  the transformation with $\lambda_2$ will change the terms of order $2$ into the sum of a positive operator of the same order plus a remainder of order $1$; 
\\
- the transformation with $\lambda_1$ will not change the terms of order $2$, but it will turn the terms of order $1$ into the sum of a positive operator of order $1$ plus a remainder of order $0$, with some growth with respect to $x$;
\\
- the transformation with $k(t)\langle x \rangle_{h}^{1-\sigma}$ will correct this remainder term, making it positive.\\
The precise definitions of $\lambda_2$ and $\lambda_1$ will be given in Section \ref{section_definition_and_estimates_of_lambda_2_lambda_1}. Since $\Lambda$ admits an algebraic growth on the $x$ variable, then $e^{\Lambda}$ presents an exponential growth; this is the reason why we need to work with pseudodifferential operators of infinite order.

The paper is organized as follows. In Section \ref{sec2} we recall some basic definitions and properties of Gelfand-Shilov spaces and the calculus for pseudodifferential operators of infinite order that we will use in the next sections. Section \ref{section_spectral_invariance_for_SG_with_Gevrey_estimates} is devoted to prove a result of spectral invariance for pseudodifferential operators with Gevrey regular symbols which is new in the literature and interesting per se. In this paper the spectral invariance will be used to prove the continuity properties of the inverse $\{e^{\Lambda}(t,x,D)\}^{-1}$.  In Section \ref{section_definition_and_estimates_of_lambda_2_lambda_1} we introduce the functions $\lambda_1,\lambda_2$ mentioned above and prove the invertibility of the operator $e^{\tilde{\Lambda}}(x,D), \tilde{\Lambda}=\lambda_1+\lambda_2$. In Section \ref{section_conjugation_of_iP} we perform the change of variable and the conjugation of the operator $P$. Section \ref{section_choices_of_M1_M2_kt} concerns the choice of the parameters appearing in the definition of $\Lambda$ in order to obtain  a positive operator on $L^2(\R)$. Finally, in Section \ref{section_proof_of_main_theorem}, we give the proof of Theorem \ref{theorem_main_theorem}.

\section{Gelfand-Shilov spaces and pseudodifferential operators of infinite order on $\R^n$}\label{sec2}
\subsection{Gelfand-Shilov spaces}\label{section_gelfand_shilov_spaces}

Given $s, \theta \geq 1$ and $A, B > 0$ we say that a smooth function $f$ belongs to $\mathcal{S}^{\theta, A}_{s, B} (\R^{n})$ if there is a constant $C > 0$  such that 
$$
|x^{\beta} \partial^{\alpha}_{x}f(x)| \leq C A^{|\alpha|} B^{|\beta|}\alpha!^{\theta} \beta!^{s}, 
$$
for every $\alpha, \beta \in \N^{n}_{0}$ and $x \in \R^{n}$. The norm 
$$
\| f \|_{\theta, s, A, B} \,= \sup_{\overset{x \in \R^{n}}{\alpha, \beta \in \N^{n}_{0}}} |x^{\beta} \partial^{\alpha}_{x}f(x)|A^{-|\alpha|}B^{-|\beta|}\alpha!^{-\theta}\beta!^{-s} 
$$
turns $\mathcal{S}^{\theta, A}_{s, B}(\R^{n})$ into a Banach space. We define 
$$
\mathcal{S}^{\theta}_{s}(\R^{n}) = \bigcup_{A,B >0} \mathcal{S}^{\theta, A}_{s, B} (\R^{n})
$$ 
and we can equip it with the inductive limit topology of the Banach spaces $\mathcal{S}^{\theta, A}_{s, B}(\R^{n})$. The spaces $\mathcal{S}^{\theta}_{s}(\R^{n})$ have been originally introduced in the book \cite{GS2}, see \cite{Pi}. We also consider the projective version, that is 
$$
\Sigma^{\theta}_{s} (\R^{n}) = \bigcap_{A, B > 0} \mathcal{S}^{\theta, A}_{s, B} (\R^{n})
$$
equipped with the projective limit topology. When $\theta = s$ we simply write $\mathcal{S}_{\theta}$, $\Sigma_{\theta}$ instead of $\mathcal{S}^{\theta}_{\theta}, \Sigma^{\theta}_{\theta}$. We can also define, for $C, \varepsilon > 0$, the normed space $\mathcal{S}^{\theta, \varepsilon}_{s, C}(\R^{n})$ given by the functions $f \in C^\infty(\R^n)$ such that there is $C > 0$ satisfying
$$
\|f\|_{s,\theta}^{\varepsilon, C}:= \sup_{\stackrel{x \in \R^{n}}{\alpha \in \N^{n}_{0}}}C^{-|\alpha|} \alpha!^{-\theta} e^{\varepsilon |x|^{\frac{1}{s}}}|\partial^{\alpha}_{x}f(x)| <\infty ,
$$
and we have (with equivalent topologies) 
$$
\mathcal{S}^{\theta}_{s}(\R^{n}) = \bigcup_{C,\varepsilon > 0} \mathcal{S}^{\theta, \varepsilon}_{s, C}(\R^{n}) \quad \text{and} \quad \Sigma^{\theta}_{s}(\R^{n}) = \bigcap_{C,\varepsilon > 0} \mathcal{S}^{\theta, \varepsilon}_{s, C}(\R^{n}).
$$

The following inclusions are continuous (for every $\varepsilon > 0$)
$$
\Sigma^{\theta}_{s} (\R^{n}) 
\subset \mathcal{S}^{\theta}_{s}(\R^{n}) \subset \Sigma^{\theta+\varepsilon}_{s+\varepsilon} (\R^{n}).
$$
All the previous spaces can be written in terms of the Gelfand-Shilov Sobolev spaces $H^m_{\rho;s,\theta},$ with $\rho, m \in \R^2$ defined in the Introduction. Namely, we have
$$\mathcal{S}^{\theta}_{s}(\R^{n})= \bigcup_{\stackrel{\rho \in \R^2}{\rho_j >0, j=1,2}}H^m_{\rho;s,\theta}(\R^{n}), \qquad \Sigma^{\theta}_{s}(\R^{n})= \bigcap_{\stackrel{\rho \in \R^2}{\rho_j >0, j=1,2}}H^m_{\rho;s,\theta}(\R^{n}).$$

From now and on we shall denote by $(\mathcal{S}^{\theta}_{s})' (\R^{n})$, $(\Sigma^{\theta}_{s})' (\R^{n})$  the respective dual spaces.

Concerning the action of the Fourier transform we have the following isomorphisms
$$
\mathcal{F}: \Sigma^{\theta}_{s} (\R^{n}) \to \Sigma^{s}_{\theta} (\R^{n}), \quad 
\mathcal{F}: \mathcal{S}^{\theta}_{s}(\R^{n}) \to \mathcal{S}^{s}_{\theta}(\R^{n}),
$$
$$\mathcal{F}: H^{(m_1,m_2)}_{(\rho_1,\rho_2);s,\theta}(\R^{n}) \to H^{(m_2,m_1)}_{(\rho_2,\rho_1);\theta,s}(\R^{n}).$$

\subsection{Pseudodifferential operators of infinite order }\label{section_pseudodifferential_operators_of_infinite_order}

We start defining the symbol classes of infinite order.

\begin{definition}\label{definition_SG_symbol_classes_infinte_order}
	Let $\tau \in \R$, $\kappa, \theta, \mu, \nu > 1$ and $C ,c> 0$.
	\begin{itemize}
		\item [(i)] We denote by $\SG^{\tau, \infty}_{\mu,\nu;\kappa}(\R^{2n}; C, c)$ the Banach space of all functions
		$p\in C^\infty(\R^{2n})$ satisfying the following condition:
		$$
\|p\|_{C,c}:=\sup_{\alpha, \beta \in \N^n_0}	C^{-|\alpha + \beta|} \alpha!^{-\mu} \beta!^{-\nu}\sup_{x,\xi \in \R^n}\langle \xi \rangle^{-\tau+|\alpha|}\langle x \rangle^{|\beta|}e^{-c|x|^{\frac1{\kappa}}}|\partial^{\alpha}_{\xi}\partial^{\beta}_{x}p(x,\xi)| <\infty.
		$$
		We set $\SG^{\tau, \infty}_{\mu,\nu;\kappa}(\R^{2n}):=\bigcup_{C,c>0}\SG^{\tau, \infty}_{\mu,\nu;\kappa}(\R^{2n}; C, c)$ with the topology of inductive limit of the Banach spaces 
		$\SG^{\tau, \infty}_{\mu,\nu;\kappa}(\R^{2n}; C,c)$.
		
		\item [(ii)]We denote by $\SG^{\infty, \tau}_{\mu,\nu;\theta}(\R^{2n}; C,c)$ the Banach space of all functions $p\in C^\infty(\R^{2n})$ satisfying the following condition:
		$$
		\|p\|^{C,c}:=\sup_{\alpha, \beta \in \N^{n}_{0}}C^{-|\alpha + \beta|} \alpha!^{-\mu} \beta!^{-\nu}\sup_{x, \xi \in \R^{n}} \langle \xi \rangle^{|\alpha|}  \langle x \rangle^{-\tau + |\beta|}e^{-c|\xi|^{\frac{1}{\theta}}} |\partial^{\alpha}_{\xi}\partial^{\beta}_{x}p(x,\xi)| <\infty .
		$$
		We set $\SG^{\infty, \tau}_{\mu,\nu; \theta}(\R^{2n}) :=\bigcup_{C,c>0} \SG^{\infty, \tau}_{\mu,\nu;\theta}(\R^{2n}; C,c)$ with the topology of inductive limit of the spaces 
		$\SG^{\infty, \tau}_{\mu,\nu; \theta}(\R^{2n}; C,c)$.
	\end{itemize} 
\end{definition}

We also need the following symbol classes of finite order.

\begin{definition}\label{definition_SG_symbol_classe_finite_order}
	Let $\mu, \nu \geq 1$, $m = (m_1, m_2) \in \R^{2}$ and $C > 0$. We denote by $\SG^{m}_{\mu,\nu}(\R^{2n}; C)$ the Banach space of all functions $p\in C^\infty(\R^{2n})$ satisfying the following condition:
	$$
	\sup_{\alpha, \beta \in \N^{n}_{0}}C^{-|\alpha + \beta|} \alpha!^{-\mu} \beta!^{-\nu}\sup_{x, \xi \in \R^{n}} \langle \xi \rangle^{-m_1+|\alpha|}  \langle x \rangle^{-m_2 + |\beta|}|\partial^{\alpha}_{\xi} \partial^{\beta}_{x}p(x,\xi)| <\infty.
	$$
We set $\SG^{m}_{\mu, \nu}(\R^{2n}):=\bigcup_{C>0}\SG^{m}_{\mu,\nu}(\R^{2n}; C)$.
	
	Finally we say that $p \in \SG^{m}(\R^{2n})$ if for any $\alpha, \beta \in \N^{n}_{0}$ there is $C_{\alpha, \beta} > 0$ satisfying
	$$
	|\partial^{\alpha}_{\xi} \partial^{\beta}_{x}p(x,\xi)| \leq C_{\alpha,\beta} 
	\langle \xi \rangle^{m_1 -|\alpha|} \langle x \rangle^{m_2 - |\beta|}, 
	\quad x,\xi \in \R^{n}.
	$$	
\end{definition}

When $\mu = \nu$ we write $\SG^{m}_{\mu}(\R^{2n})$, $\SG^{\tau,\infty}_{\mu,\kappa}(\R^{2n})$, $\SG^{\infty,\tau}_{\mu,\theta}(\R^{2n})$ instead of $\SG^{m}_{\mu, \mu}(\R^{2n})$, $\SG^{\tau, \infty}_{\mu,\mu; k}(\R^{2n})$, $\SG^{\infty, \tau}_{\mu, \mu; \theta}(\R^{2n})$. 

As usual, given a symbol $p(x,\xi)$ we shall denote by $p(x,D)$ or by $\textrm{op} (p)$ the pseudodifferential operator defined as standard by 
$$
p(x,D) u (x) = \int e^{i x \cdot \xi} p(x,\xi) \widehat{u}(\xi) \dslash \xi, \quad x \in \R^{n},
$$
where $u$ belongs to some suitable function space depending on the assumptions on $p$, and $\dslash\xi$ stands for $(2\pi)^{-n}d\xi$. 
We have the following continuity results.

\begin{proposition}\label{proposition_continuity_pseudo_infinite_order_on_x_Gelfand_shilov_spaces}
	Let $\tau \in \R$, $s > \mu \geq 1$, $\nu \geq 1$ and $p \in \SG^{\tau, \infty}_{\mu, \nu; s}(\R^{2n})$. Then for every $\theta >\nu$ the pseudodifferential operator $p(x,D)$ with symbol $p(x,\xi)$ is continuous on $\Sigma^{\theta}_{s}(\R^{n})$ and it extends to a continuous map on $(\Sigma^{\theta}_{s})'(\R^{n})$.
\end{proposition}

\begin{proposition}\label{proposition_continuity_pseudo_infinite_order_on_xi_Gelfand_shilov_spaces}
	Let $\tau \in \R$, $\theta > \nu \geq 1$, $\mu \geq 1$ and $p \in \SG^{\infty, \tau}_{\mu, \nu; \theta}(\R^{2n})$. Then for every $s >\mu$ the operator $p(x,D)$ is continuous on $\Sigma^{\theta}_{s}(\R^{n})$ and it extends to a continuous map on $(\Sigma^{\theta}_{s})'(\R^{n})$).
\end{proposition}

The proof of Propositions \ref{proposition_continuity_pseudo_infinite_order_on_x_Gelfand_shilov_spaces} and \ref{proposition_continuity_pseudo_infinite_order_on_xi_Gelfand_shilov_spaces} can be derived following the argument in the proof of \cite[Proposition 2.3]{ascanelli_cappiello_schrodinger_equations_Gelfand_shillov} and \cite[Proposition 1]{AAC3evolGevrey}. We leave details to the reader.

Now we define the notion of asymptotic expansion and recall some fundamental results, which can be found in the Appendix A of \cite{ascanelli_cappiello_schrodinger_equations_Gelfand_shillov}. For $t_1, t_2 \geq 0$ set 
$$
Q_{t_1,t_2} = \{(x,\xi) \in \R^{2n} : \langle x \rangle < t_1 \,\, \text{and} \,\, \langle \xi \rangle < t_2 \}
$$
and $Q^{e}_{t_1, t_2} = \R^{2n} \setminus Q_{t_1, t_2}$. When $t_1 = t_2 = t$ we simply write $Q_t$ and $Q^{e}_{t}$.

\begin{definition}\label{definition_classes_of_asymptotic_expansions} We say that: 
	\begin{itemize}
		\item [(i)] $\sum\limits_{j \geq 0} a_j \in FSG^{\tau, \infty}_{\mu, \nu; \kappa}$ if $a_j(x, \xi) \in C^{\infty}(\R^{2n})$ and there are $C, c, B > 0$ satisfying
		$$
		|\partial^{\alpha}_{\xi}\partial^{\beta}_{x} a_j(x, \xi)| \leq C^{|\alpha| + |\beta| + 2j + 1} \alpha!^{\mu} \beta!^{\nu} j!^{\mu + \nu -1} 
		\langle \xi \rangle^{\tau - |\alpha| - j} \langle x \rangle^{-|\beta| - j} e^{c|x|^{\frac{1}{\kappa}}}, 
		$$
		for every $\alpha, \beta \in \N^{n}_{0}$, $j \geq 0$ and $(x, \xi) \in Q^{e}_{B(j)}$, where $B(j) := Bj^{\mu + \nu - 1}$;
		
		\item [(ii)] $\sum_{j \geq 0} a_j \in FSG^{\infty, \tau}_{\mu, \nu ; \theta}$ if $a_j(x, \xi) \in C^{\infty}(\R^{2n})$ and there are $C, c, B > 0$ satisfying
		$$
		|\partial^{\alpha}_{\xi}\partial^{\beta}_{x} a_j(x, \xi)| \leq C^{|\alpha| + |\beta| + 2j + 1} \alpha!^{\mu} \beta!^{\nu} j!^{\mu + \nu -1} 
		\langle \xi \rangle^{- |\alpha| - j}  e^{c|\xi|^{\frac{1}{\theta}}} \langle x \rangle^{ \tau -|\beta| - j},
		$$
		for every $\alpha, \beta \in \N^{n}_{0}$, $j \geq 0$ and $(x, \xi) \in Q^{e}_{B(j)}$;
		
		\item [(iii)] $\sum\limits_{j \geq 0} a_j \in FSG^{m}_{\mu, \nu}$ if $a_j(x, \xi) \in C^{\infty}(\R^{2n})$ and there are $C,  B > 0$ satisfying
		$$
		|\partial^{\alpha}_{\xi}\partial^{\beta}_{x} a_j(x, \xi)| \leq C^{|\alpha| + |\beta| + 2j + 1} \alpha!^{\mu} \beta!^{\nu} j!^{\mu + \nu -1} 
		\langle \xi \rangle^{m_1 - |\alpha| - j} \langle x \rangle^{m_2 -|\beta| - j}, 
		$$
		for every $\alpha, \beta \in \N^{n}_{0}$, $j \geq 0$ and $(x, \xi) \in Q^{e}_{B(j)}$.
	\end{itemize}	   
\end{definition}

\begin{definition}
	Let $\sum\limits_{j \geq 0} a_j$, $\sum\limits_{j \geq 0} b_j$ in $FSG^{\tau, \infty}_{\mu, \nu; \kappa}$. We say that $\sum\limits_{j \geq 0} a_j \sim \sum\limits_{j \geq 0} b_j$ in $FSG^{\tau, \infty}_{\mu, \nu ; \kappa}$ if  there are $C, c, B > 0$ satisfying 
	$$
	|\partial^{\alpha}_{\xi}\partial^{\beta}_{x} \sum_{j < N} (a_j - b_j) (x, \xi)| \leq C^{|\alpha| + |\beta| + 2N + 1} \alpha!^{\mu} \beta!^{\nu} N!^{\mu + \nu - 1} 
	\langle \xi \rangle^{\tau - |\alpha| - N} \langle x \rangle^{-|\beta| - N} e^{c|x|^{\frac{1}{\kappa}}}, 
	$$
	for every $\alpha, \beta \in \N^{n}_{0}$, $N \geq 1$ and $(x, \xi) \in Q^{e}_{B(N)}$. Analogous definitions for the classes $FSG^{\infty, \tau}_{\mu, \nu; \theta}$, $FSG^{m}_{\mu, \nu}$.
\end{definition}

\begin{remark}
	If $\sum\limits_{j \geq 0} a_j \in FSG^{\tau, \infty}_{\mu, \nu ; \kappa}$, then $a_0 \in \SG^{\tau, \infty}_{\mu, \nu ; \kappa}$. Given $a \in \SG^{\tau, \infty}_{\mu, \nu ; \kappa}$ and setting $b_0 = a$, $b_j = 0$, $j \geq 1$, we have $a = \sum\limits_{j \geq 0}b_j$. Hence we can consider $\SG^{\tau, \infty}_{\mu, \nu ; \kappa}$ as a subset of $FSG^{\tau, \infty}_{\mu, \nu ; \kappa}$.
\end{remark} 

\begin{proposition}\label{Proposition_existence_of_a_stymbol_which_has_an_given_asymptotica_expansion}
	Given $\sum\limits_{j \geq 0} a_j \in FSG^{\tau, \infty}_{\mu, \nu ; \kappa},$ there exists $a \in \SG^{\tau, \infty}_{\mu, \nu ; \kappa}$ such that $a \sim \sum\limits_{j \geq 0} a_j$ in $FSG^{\tau, \infty}_{\mu, \nu ; \kappa}$. Analogous results for the classes $FSG^{\infty, \tau}_{\mu, \nu ; \theta}$ and $FSG^{m}_{\mu, \nu}$.
\end{proposition}

\begin{proposition}
	Let $a \in \SG^{0, \infty}_{\mu, \nu ; \kappa}$ such that $a \sim 0$ in $FSG^{0, \infty}_{\mu, \nu ; \kappa}$.	
	If $\kappa > \mu + \nu - 1$, then $a \in \mathcal{S}_\delta(\R^{2n})$ for every $\delta \geq \mu + \nu - 1$. Analogous results for the classes $FSG^{\infty, \tau}_{\mu, \nu ; \theta}$ and $FSG^{m}_{\mu, \nu}$.
\end{proposition}

Concerning the symbolic calculus and the continuous mapping properties on the Gelfand-Shilov Sobolev spaces we have the following results, cf. \cite[Propositions A.12 and A.13]{ascanelli_cappiello_schrodinger_equations_Gelfand_shillov}.

\begin{theorem}\label{theorem_symbolic_calculus_of_infinte_order}
	Let $p \in \SG^{\tau,  \infty}_{\mu, \nu; \kappa}(\R^{2n})$, $q \in \SG^{\tau', \infty}_{\mu, \nu; \kappa}(\R^{2n})$ with $\kappa > \mu + \nu - 1$. Then the $L^{2}$ adjoint $p^{*}$ and the composition $p\circ q$ have the following structure:
\begin{itemize}	
\item[] $p^{*}(x,D) = a(x, D) + r(x,D)$ where $r\in \mathcal{S}_{\mu+\nu-1}(\R^{2n})$, $a \in \SG^{\tau,  \infty}_{\mu, \nu; \kappa}(\R^{2n})$, and
	$$
	a(x,\xi) \sim \sum_{\alpha} \frac{1}{\alpha!} \partial^{\alpha}_{\xi}D^{\alpha}_x \overline{p(x,\xi)} \,\, \text{in} \,\,
	FSG^{\tau,  \infty}_{\mu, \nu; \kappa}(\R^{2n});
	$$
\item[] $p(x,D)\circ q(x, D) = b(x,D) + s(x,D)$, where $s \in \mathcal{S}_{\mu+\nu-1}(\R^{2n})$, $b \in \SG^{\tau+\tau',  \infty}_{\mu, \nu; \kappa}(\R^{2n})$ and
	$$
	b(x, \xi) \sim \sum_{\alpha} \frac{1}{\alpha!} \partial^{\alpha}_{\xi}p(x,\xi) D^{\alpha}_xq(x,\xi) \,\, \text{in} \,\,
	FSG^{\tau+\tau',  \infty}_{\mu, \nu; \kappa}(\R^{2n}).
	$$
	\end{itemize}
	Analogous results for the classes $\SG^{\infty, \tau}_{\mu, \nu; \theta}(\R^{2n})$ and $\SG^{m}_{\mu,\nu}(\R^{2n})$.
\end{theorem}

\begin{theorem}\label{theorem_continuity_finite_order_in_gelfand_shilov_sobolev_spaces}
	Let $p \in \SG^{m'}_{\mu, \nu}(\R^{2n})$ for some $m' \in \R^{2}$. Then for every $m, \rho \in \R^{2}$ and $s,\theta$ such that $\min \{s, \theta\} > \mu + \nu - 1$ the operator $p(x, D)$ maps $H^{m}_{\rho; s, \theta}(\R^{n})$ into $H^{m-m'}_{\rho; s, \theta}(\R^{n})$ continuously.
\end{theorem}

A simple application of the Fa\`a formula gives us the following result.

\begin{proposition}\label{proposition_exponential_of_symbols_of_finite_order}
	If $\lambda \in \SG^{0,\frac{1}{\kappa}}_{\mu}(\R^{2n})$, then $e^{\lambda} \in \SG^{0, \infty}_{\mu; \kappa}(\R^{2n})$. If $\lambda \in \SG^{\frac{1}{\theta}, 0}_{\mu}(\R^{2n})$, then $e^{\lambda} \in \SG^{\infty, 0}_{\mu; \theta}(\R^{2n})$.
\end{proposition}

We conclude this section proving the following theorem.

\begin{theorem}\label{theorem_continuity_infinite_order_in_gelfand_shilov_sobolev_spaces}
	Let $\rho, m \in \R^{2}$ and $s, \theta, \mu > 1$ with $\min \{s,\theta\} > 2\mu - 1$. Let $\lambda \in \SG^{0, \frac{1}{\kappa}}_{\mu}(\R^{2n})$. Then: \\
i)
if $\kappa >s$, the operator $e^\lambda(x,D):H^{m}_{\rho; s, \theta} (\R^{n})\longrightarrow H^{m}_{\rho-\delta e_2; s, \theta}(\R^{n})$ is continuous for every $\delta >0$, where $e_2 = (0, 1)$; \\
ii) if $\kappa =s $, the operator $e^{\lambda}(x,D):H^{m}_{\rho; s, \theta} (\R^{n})\longrightarrow H^{m}_{\rho-\delta e_2; s, \theta}(\R^{n})$ is continuous for every $$\delta >C(\lambda):= \sup_{(x,\xi) \in \R^{2n}} \ds\frac{\lambda(x,\xi)}{\langle x \rangle^{1/s}}.$$
\end{theorem}
\begin{proof}
(i) Let $\phi(x)\in \gamma^\mu(\R^n)$, that is, a uniform Gevrey function of index $\mu$, be a cut-off function such that, for a positive constant $K$, $\phi(x)=1$ for $|x|<K/2$, $\phi(x)=0$ for $|x|>K$ and $0\leq\phi(x)\leq 1$ for every $x\in\R^n$. We split the symbol $e^{\lambda(x,\xi)}$ as
\beqs\label{l1}e^{\lambda(x,\xi)} = \phi(x)e^{\lambda(x,\xi)} +(1-\phi(x))e^{\lambda(x,\xi)} = a_1(x,\xi)+ a_2(x,\xi).
\eeqs
Since $\phi$ has compact support and $\lambda$ has order zero with respect to $\xi$, we have $a_1 \in \SG_\mu^{0,0}$. 
On the other hand, given any $\delta>0$ and choosing $K$ large enough, since $\kappa>s$ we may write $|\lambda(x,\xi)|\langle x\rangle^{-1/s}<\delta$ on the support of $a_2(x,\xi)$.  Hence we obtain
$$a_2(x,\xi)=e^{\delta\langle x\rangle^{1/s}}(1-\phi(x))e^{\lambda(x,\xi)-\delta\langle x\rangle^{1/s}},$$
with $(1-\phi(x))e^{\lambda(x,\xi)-\delta\langle x\rangle^{1/s}}$
of order $(0,0)$ because $\lambda(x,\xi)-\delta\langle x\rangle^{1/s}<0$ on the support of
$(1-\phi(x))$. Thus, \eqref{l1} becomes
\beqsn e^{\lambda(x,\xi)}= a_1(x,\xi)+e^{\delta\langle x\rangle^{1/s}}\tilde a_2(x,\xi),
\eeqsn
$a_1$ and $\tilde a_2$ of order $(0,0)$.
Since by Theorem \ref{theorem_continuity_finite_order_in_gelfand_shilov_sobolev_spaces} the operators $a_1(x,D)$ and $\tilde a_2(x,D)$ map continuously $H^m_{\rho,s,\theta}$ into
itself, then we obtain i). The proof of (ii) follows a similar argument and can be found in \cite[Theorem 2.4]{ascanelli_cappiello_schrodinger_equations_Gelfand_shillov}.
\end{proof}
 
\section{Spectral invariance for $SG$-$\Psi$DO with Gevrey estimates}\label{section_spectral_invariance_for_SG_with_Gevrey_estimates}

Let $p \in \SG^{0,0}(\R^{2n})$, then $p(x, D)$ extends to a continuous operator on $L^{2}(\R^{n})$. Suppose that $p(x,D): L^{2}(\R^{n}) \to L^{2}(\R^{n})$ is bijective. The question is to determine whether or not the inverse $p^{-1}$ is also an $\SG$ operator of order $(0,0)$. This is known as the spectral invariance problem and it has an affirmative answer, see \cite{dasgupta_wong_spectral_invariance_SG_pseduo_diff_operators}. 

Following the ideas presented in \cite[pp. 51-57]{dasgupta_wong_spectral_invariance_SG_pseduo_diff_operators}, we will prove that the symbol of $p^{-1}$ satisfies Gevrey estimates, whenever the symbol $p \in \SG^{0,0}_{\mu, \nu}(\R^{2n})$. This is an important step in the study of the continuous mapping properties of $\{e^{\Lambda}(x,D)\}^{-1}$ on Gelfand-Shilov Sobolev spaces $H^m_{\rho;s,\theta}$.

Theorems \ref{theorem_aktinson}, \ref{theorem_composition_fredholm_operators}, \ref{theorem_equivalence_ellipticity_fredholmness} here below can be found in \cite[Chapters 20, 21]{wong_book_introduction_to_pseudo_diff_operators}.

\begin{theorem}\label{theorem_aktinson}
	Let $X, Y$ separable Hilbert spaces. Then a bounded operator $A: X \to Y$ is Fredholm if and only if there are $B: Y \to X$ bounded, $K_1: X \to Y$ and $K_2: Y \to X$ compact operators such that
	$$
	BA = I_X - K_1, \qquad AB = I_{Y} - K_2.
	$$ 
\end{theorem}

\begin{theorem}\label{theorem_composition_fredholm_operators}
	Let $X, Y, Z$ be separable Hilbert spaces and let $A: X \to Y$, $B: Y \to Z$ be Fredholm operators. Then: 
	\begin{itemize}
		\item $B \circ A : X \to Z$ is Fredholm and $i(BA) = i(B) + i(A)$, where $i(\cdot)$ stands for the index of a Fredholm operator;
		\item $Y = N(A^{t}) \oplus R(A)$.
	\end{itemize}
\end{theorem}

\begin{remark}\label{remark_identity_minus_compact_Fredholm}
	Let $X$ be a Hilbert space and $K: X \to X$ a compact operator. Then $I-K$ is Fredholm and $i(I-K) = 0$.
\end{remark}

\begin{theorem}\label{theorem_equivalence_ellipticity_fredholmness}
	Let $p \in \SG^{m_1, m_2}(\R^{2n})$ such that $p(x,D): H^{s_1+m_1,s_2+m_2}(\R^{n}) \to H^{m_1,m_2}(\R^{n})$ is Fredholm for some $s_1, s_2 \in \R$. Then $p$ is $\SG$-elliptic, that is there exist $C,R>0$ such that
$$|p(x,\xi)| \geq C \langle \xi \rangle^{m_1} \langle x \rangle^{m_2} \qquad \textit{for}\quad (x,\xi) \in Q_R^e.$$ 
\end{theorem}

\begin{theorem}\label{theorem_parametrices_for_SG_elliptic_with_gevrey_estimates}
	Let $p \in \SG^{m_1, m_2}_{\mu,\nu}(\R^{2n})$ be $\SG$-elliptic. Then there is $q \in \SG^{-m_1, -m_2}_{\mu,\nu}(\R^{2n})$ such that 
	$$
	p(x,D) \circ q(x,D) = I + r_1(x,D), \qquad q(x,D) \circ p(x,D) = I + r_2(x,D),
	$$
	where $r_1, r_2 \in S_{\mu+\nu-1}(\R^{2n})$.
\end{theorem}
\begin{proof}
	See \cite[Theorem 6.3.16]{nicola_rodino_global_pseudo_diffferential_calculus_on_euclidean_spaces}.
\end{proof}

In order to prove the main result of this section, we need the following technical lemma.

\begin{lemma}\label{lemma_A_given_by_kernel_in_projective_gelfand_shilov}
	Let $A : L^{2}(\R^n) \to L^{2}(\R^n)$ be a bounded operator such that $A$ and $A^{*}$ map $L^2(\R^n)$ into $\Sigma_{r}(\R^n)$ continuously. Then the Schwartz kernel of $A$ belongs to $\Sigma_{r}(\R^{2n})$.
\end{lemma}

\begin{proof}
	Since $\Sigma_{r}(\R^{n}) \subset L^{2}(\R^{n})$ is a nuclear Fr\'echet space, (cf. \cite{debrouwere_nuclearity_gelfand_shilov}), by \cite[Propositions 2.1.7 and 2.1.8 ]{harutyunyan_sculze_book_kernel_theorems}, we have that $A$ is defined by a kernel $H(x,y)$ and we can write 
	$$
	H(x,y) = \sum_{j \in \N_{0}} a_j f_j(x) g_j(y) = \sum_{j \in \N_{0}} \tilde{a}_{j} \tilde{f}_{j}(x) \tilde{g}_{j}(y),
	$$  
	where $a_{j}, \tilde{a}_{j} \in \C$, $\tilde{f}_{j}(x), g_{j}(y) \in \Sigma_{r}(\R^{n})$, $f_{j}(x), \tilde{g}_{j}(y) \in L^{2}(\R^{n})$, $\sum_{j}|a_j| < \infty$, $\sum_{j}|\tilde{a}_{j}| < \infty$, $\tilde{f}_{j}(x), g_{j}(y)$ converge to zero in $\Sigma_{r}(\R^{n})$ and $f_{j}(x), \tilde{g}_{j}(y)$ converge to zero in $L^{2}(\R^{n})$.
	
	We now use the following characterization: $H \in \Sigma_{r}(\R^{2n})$ if and only if 
	$$
	\sup_{\alpha, \beta \in \N^{n}_{0}} \left\| \dfrac{x^{\alpha}y^{\beta}H(x,y)}{C^{|\alpha|+|\beta|}\alpha!^r\beta!^r} \right\|_{L^{2}}  < \infty
	\quad \text{and} \quad
	\sup_{\alpha, \beta \in \N^{n}_{0}} \left\| \dfrac{\xi^{\alpha}\eta^{\beta}\widehat{H}(\xi,\eta)}{C^{|\alpha|+|\beta|}\alpha!^r\beta!^r} \right\|_{L^{2}} < \infty
	$$
	for every $C>0$, and prove that both the latter conditions hold. Note that
	\begin{align*}
	\| y^{\beta}H(x,y) \|^{2}_{L^{2}} &= \iint \left|\sum_{j \in \N_{0}} a_j f_j(x) y^{\beta} g_j(y)\right|^{2} dx\,dy \\
	&= \iint \left|\sum_{j \in \N_{0}} a_j^{\frac{1}{2}} f_j(x) a_j^{\frac{1}{2}} y^{\beta} g_j(y)\right|^{2} dx\,dy \\
	&\leq \int  \sum_{j \in \N_{0}} |a_j^{\frac{1}{2}} f_j(x)|^{2} dx 
	\int  \sum_{j \in \N_{0}} |a_j^{\frac{1}{2}} y^{\beta} g_j(y)|^{2} dy \\
	&= \sum_{j \in \N_{0}} |a_j| \|f_j\|^{2}_{L^{2}} \sum_{j \in \N_{0}} |a_j| \|y^{\beta}g_j\|^{2}_{L^{2}}.
	\end{align*}
	Since $g_j$ converges to zero in $\Sigma_r(\R^{n})$, we have 
	$$
	\| y^{\beta} g_{j}(y) \|_{L^2} = \left\| \dfrac{y^{\beta} g_{j}(y)}{C^{|\beta|}\beta!^{r}} \right\|_{L^2} C^{|\beta|} \beta!^{r} \leq C^{|\beta|} \beta!^{r} \sup_{j \in \N_{0}} \left\| \dfrac{y^{\beta} g_{j}(y)}{C^{|\beta|}\beta!^{r}} \right\|_{L^2},
	$$ 
	for every $C>0$, and therefore 
	$$
	\| y^{\beta}H(x,y) \|_{L^{2}} \leq \left(\sum_{j \in \N_0} |a_j|\right) \sup_{j \in \N_{0}} \| f_j \|_{L^2} 
	\sup_{j \in \N_{0}} \left\| \dfrac{y^{\beta} g_{j}(y)}{C^{|\beta|}\beta!^{r}} \right\|_{L^2} C^{|\beta|} \beta!^{r}. 
	$$
	Hence 
	$$
	\sup_{\beta\in \N^{n}_{0}}\ \left\| \dfrac{y^{\beta} H(x,y)}{C^{|\beta|} \beta!^{r}} \right\|_{L^2} < \infty
	\iff
	\sup_{N \in \N_{0}}\ \left\| \dfrac{ \langle y \rangle^{N} H(x,y)}{C^{N} N!^{r}} \right\|_{L^2} < \infty,
	$$
	for every $C >0$. Using the representation $\sum \tilde{a}_{j} \tilde{f}_{j}(x) \tilde{g}_{j}(y)$, analogously we can obtain
	$$
	\sup_{\alpha \in \N^{n}_{0}}\ \left\| \dfrac{x^{\alpha} H(x,y)}{C^{|\alpha|} \alpha!^{r}} \right\|_{L^2} < \infty
	\iff
	\sup_{N \in \N_{0}}\ \left\| \dfrac{ \langle x \rangle^{N} H(x,y)}{C^{N} N!^{r}} \right\|_{L^2} < \infty,
	$$
	for every $C>0$.
	
	Now note that, for every $N \in \N_0$, $x, y \in \R^{n}$, 
	$$
	\langle x, y \rangle^{N} = (\langle x \rangle^{2} + |y|^{2})^{\frac{N}{2}} \leq (\langle x \rangle + \langle y \rangle)^{N} \leq 2^{N-1}(\langle x \rangle^N + \langle y \rangle^N).$$
	Therefore, for every $C > 0$, 
	$$
	\left \| \frac{\langle x, y \rangle^{N} H(x, y)}{C^{N} N!^{r}} \right\|_{L^{2}} \leq 
 \left\| \frac{\langle x \rangle^{N} H(x, y)}{C^{N}_{1} N!^{r}} \right\|_{L^{2}} + 
	\left\| \frac{\langle y \rangle^{N} H(x, y)}{C^{N}_{1} N!^{r}} \right\|_{L^{2}},
	$$
	where $C_1 = (2^{-1}C)^{N}$. Hence, for every $C > 0$,
	$$
	\sup_{N \in \N_0} \left\| \frac{\langle x, y \rangle^{N} H(x, y)}{C^{N} N!^{r}} \right\|_{L^{2}} < \infty.
	$$
	
	Since the Fourier transformation is an isomorphism on $L^{2}$ and on $\Sigma_{r}$, we have 
	$$
	\widehat{H}(\xi,\eta) = \sum_{j \in \N_{0}} a_j \widehat{f}_j(\xi) \widehat{g}_j(\eta) = \sum_{j \in \N_{0}} \tilde{a}_{j} \widehat{\tilde{f}}_{j}(\xi) \widehat{\tilde{g}}_{j}(\eta),
	$$  
	where $a_{j}, \tilde{a}_{j} \in \C$, $\widehat{\tilde{f}}_{j}(\xi), \widehat{g}_j(\eta) \in \Sigma_{r}(\R^{n})$, $\widehat{f}_j(\xi), \widehat{\tilde{g}}_{j}(\eta) \in L^{2}(\R^{n})$, $\sum_{j}|a_j| < \infty$, $\sum_{j}|\tilde{a}_{j}| < \infty$, $\widehat{\tilde{f}}_{j}(\xi), \widehat{g}_j(\eta)$ converge to zero in $\Sigma_{r}(\R^{n})$ and $\widehat{f}_j(\xi), \widehat{\tilde{g}}_{j}(\eta)$ converge to zero in $L^{2}(\R^{n})$. In an analogous way as before we get, for every $C > 0$,
	$$
	\sup_{N \in \N_0} \left\| \frac{\langle \xi, \eta \rangle^{N} \widehat{H}(\xi, \eta)}{C^{N} N!^{r}} \right\|_{L^{2}} < \infty.
	$$
	Hence $H \in \Sigma_{r}(\R^{2n})$.
\end{proof}

\begin{theorem}\label{theorem_spectral_invariance_SG_operator_with_gevrey_estimates}
	Let $p \in \SG^{0,0}_{\mu,\nu}(\R^{2n})$ such that $p(x,D): L^2(\R^n) \to L^2(\R^n)$ is bijective. Then $\{p(x,D)\}^{-1}: L^2(\R^n) \to L^2(\R^n)$ is a pseudodifferential operator given by a symbol $\tilde{p}= q + \tilde{k}$ where 
	$q \in \SG^{0,0}_{\mu,\nu}(\R^{2n})$ and $\tilde{k} \in \Sigma_{r}(\R^{2n})$ for every $r > \mu + \nu - 1$.
\end{theorem}
\begin{proof}
	Since $p(x,D): L^2(\R^n) \to L^2(\R^n)$ is bijective, then $p(x, D)$ is Fredholm and 
	$$
	i(p(x, D)) = dim N(p(x,D)) - dim N(p^{t}(x,D)) = 0.
	$$
	Therefore by Theorem \ref{theorem_equivalence_ellipticity_fredholmness} $p$ is $\SG$-elliptic and by Theorem \ref{theorem_parametrices_for_SG_elliptic_with_gevrey_estimates} there is $q \in \SG^{0,0}_{\mu, \nu}(\R^{2n})$ such that
	$$
	q(x,D) \circ p(x,D) = I + r(x,D), \qquad p(x,D) \circ q(x,D) = I + s(x,D),
	$$
	for some $r, s \in \mathcal{S}_{\mu+\nu-1}(\R^{2n})$.
	In particular $r(x,D), s(x,D)$ are compact operators on $ L^2(\R^n)$. By Theorem \ref{theorem_aktinson} $q(x,D)$ is a Fredholm operator and we have 
	$$
	i(q(x,D)) = i(q(x,D)) + i(p(x,D)) = i(q(x,D) \circ p(x,D)) = i(I + r(x,D)) = 0.
	$$  
	
	Note that $N(q(x,D))$ and $N(q^{t}(x,D))$ are subspaces of $\mathcal{S}_{\mu+\nu-1}(\R^{n})$. Indeed, let $f \in N(q)$ and $g \in N(q^t)$, then 
	$$
	0 = p(x,D) \circ q(x,D) f = (I+s(x,D))f \implies f = -s(x,D)f,
	$$
	$$
	0 = p^t(x,D) \circ q^t(x,D) g = (q(x,D) \circ p(x,D))^t g = (I+r(x,D))^t g \implies g = -r^{t}(x,D)g. 
	$$
	
	Since $L^2(\R^n)$ is a separable Hilbert space and $N(q(x,D))$ is closed, we have the following decompositions 
	$$
	L^2 = N(q) \oplus N(q)^{\bot}, \qquad L^2 = N(q^t) \oplus R_{L^2}(q),
	$$
	where $R_{L^2}(q)$ denotes the range of $q(x,D)$ as an operator on $L^2(\R^{n})$.
	
	Let $\pi: L^2 \to N(q)$ the projection of $L^2$ onto $N(q)$ with null space $N(q)^{\bot}$, $F: N(q) \to N(q^t)$ an isomorphism and $i: N(q^t) \to L^2$ the inclusion. Set $Q = i\circ F \circ \pi$. Then $Q: L^2 \to L^2$ is bounded and its image is contained in $N(q^t) \subset \mathcal{S}_{\mu+\nu-1}$. It is not difficult to see that $Q^{*} = \tilde{i} \circ F^{*} \circ \pi_{N(q^{t})}$, where $\tilde{i}$ is the inclusion of $N(q)$ into $L^2$ and $\pi_{N(q^{t})}$ is the orthogonal projection of $L^2$ onto $N(q)$. Since $\mathcal{S}_{\mu+\nu-1} \subset \Sigma_{r}$, then  by Lemma \ref{lemma_A_given_by_kernel_in_projective_gelfand_shilov}, $Q$ is given by a kernel in $\Sigma_{r}$.
	
	We will now show that $q + Q$ is a bijective parametrix of $p$. Indeed, let $u = u_1 + u_2 \in N(q)\oplus N(q^t)$ such that $(q + Q) u = 0$. Then $0 = q u_2 + (i\circ F)u_1 \in R_{L^2}(q) \oplus N(q^t)$. Hence $q u_2 = 0$ and $i\circ F u_1 = 0$ which implies that $u = 0$. In order to prove that $Q$ is surjective, consider $f = f_1 + f_2 \in R_{L^2}(q) \oplus N(q^t)$. There exist $u_1 \in L^2$ and $u_2 \in N(q)$ such that $q u_1 = f_1$ and $F u_2 = f_2$. Now write $u_1 = v_1 + v_2 \in N(q) \oplus N(q)^{\bot}$. Then $q(u_1) = q(v_2)$ and therefore $(q+Q)(v_2+u_2) = f_1 + f_2 = f$. Finally notice that 
	$$
	p(x,D) \circ (q(x,D) + Q) = I + s(x,D) + p(x,D) \circ Q = I + s'(x,D),
	$$
	$$
	(q(x,D) + Q) \circ p(x,D) = I + r(x,D) + Q \circ p(x,D) = I + r'(x,D),
	$$
	where $r', s' \in \Sigma_{r}(\R^{2n})$.
	
	Now set $\tilde{q} = q(x,D) + Q$. Therefore $\tilde{q} \circ p(x,D) = I + r'(x,D): L^2 \to L^2$ is bijective. Set $k = -(I+r')^{-1} \circ r'$. Then $(I+r')(I+k) = I$ and $k= - r' - r'k$. Observe that 
	$$
	k^t = -\{r'\}^t - k^t\{r'\}^t = -\{r'\}^t+\{r'\}^t\{(I+r')^{-1}\}^{t}\{r'\}^t.
	$$ 
	Hence $k, k^t$ map $L^2$ into $\Sigma_{r}$ and by Lemma \ref{lemma_A_given_by_kernel_in_projective_gelfand_shilov} we have that $k$ is given by a kernel in $\Sigma_{r}(\R^{2n})$.
	
	To finish the proof, it is enough to notice that 
	$$
	p^{-1} \circ \tilde{q}^{-1}  = (\tilde{q} \circ p)^{-1} = (I+r')^{-1} = I+k \implies p^{-1} = (I+k) \circ \tilde{q}.
	$$
	
\end{proof}


\section{Change of variables}
\label{section_definition_and_estimates_of_lambda_2_lambda_1}
\subsection{Definition and properties of $\lambda_2$ and $\lambda_1$}

Let  $M_2, M_1>0$ and $h \geq 1$ to be chosen later on. We define 
\begin{equation}\label{equation_definition_of_lambda_2}
\lambda_2(x, \xi) = M_2 w\left(\frac{\xi}{h}\right) \int_{0}^{x} \langle y \rangle ^{-\sigma} dy, \quad (x, \xi) \in \R^2,
\end{equation}
\begin{equation}\label{equation_definition_of_lambda_1}
\lambda_1(x, \xi) = M_1 w\left(\frac{\xi}{h}\right) \langle \xi \rangle^{-1}_{h} \int_{0}^{x} \langle y \rangle ^{-\frac{\sigma}{2}} \psi\left(\frac{\langle y \rangle^{\sigma}}{\langle \xi \rangle^{2}_{h}}\right) dy, \quad (x, \xi) \in \R^2,
\end{equation}
where 
$$
w (\xi) =
\begin{cases}
0, \qquad \qquad \qquad \,\,\, |\xi| \leq 1\\
-\text{sgn}(\partial_{\xi}a_3(t, \xi)) , \quad |\xi| > R_{a_3}
\end{cases},
\quad 
\psi (y) =
\begin{cases}
1, \quad |y| \leq \frac{1}{2} \\
0, \quad |y| \geq 1
\end{cases},
$$	
$|\partial^{\alpha} w(\xi)| \leq C_{w}^{\alpha + 1} \alpha!^{\mu}$, $|\partial^{\beta} \psi(y)| \leq C_{\psi}^{\beta + 1}\beta!^{\mu}$ for some $\mu > 1$ which will be chosen later. Notice that by the assumption (i) of Theorem \ref{theorem_main_theorem} the function $w(\xi)$ is constant for $\xi \geq R_{a_3}$ and for $\xi \leq -R_{a_3}.$

\begin{lemma}\label{lemma_estimates_lambda_2}
	Let $\lambda_2(x, \xi)$ as in (\ref{equation_definition_of_lambda_2}). Then there exists $C>0$ such that for all $\alpha, \beta \in \N$ and $(x,\xi) \in \R^2$:
	\begin{itemize}
		\item[(i)] $|\lambda_2(x, \xi)| \leq \frac{M_2}{1-\sigma} \langle x \rangle^{1-\sigma};$
		\item[(ii)] $|\partial^{\beta}_{x}\lambda_2(x, \xi)| \leq M_2 C^{\beta} \beta! \langle x \rangle^{1-\sigma-\beta}$, for $\beta \geq 1$;
		\item[(iii)] $| \partial^{\alpha}_{\xi} \partial^{\beta}_{x}\lambda_2(x, \xi)| \leq M_2 C^{\alpha+\beta +1} \alpha!^{\mu} \beta! \chi_{E_{h, R_{a_3}}} (\xi) 
		\langle \xi \rangle^{-\alpha}_{h} \langle x \rangle^{1-\sigma - \beta}$, for $\alpha \geq 1, \beta \geq 0$,
	\end{itemize}
	where $E_{h, R_{a_3}} = \{ \xi \in \R \colon h \leq \xi \leq R_{a_3}h\}$. In particular $\lambda_2 \in \SG^{0, 1-\sigma}_{\mu}(\R^2)$.
\end{lemma}
\begin{proof}
	First note that
$$	|\lambda_2(x, \xi)| =  M_2 \left|w\left( \frac{\xi}{h}\right)\right| \int_{0}^{|x|} \langle y \rangle ^{-\sigma} dy 
	\leq M_2  \int_{0}^{\langle x \rangle} y^{-\sigma} dy =  \frac{M_2}{1-\sigma} \langle x \rangle^{1-\sigma}.$$
	For $\beta \geq 1$
$$	|\partial^{\beta}_{x} \lambda_2(x, \xi)| \leq M_2 \left|w\left( \frac{\xi}{h}\right)\right| |\partial^{\beta - 1}_{x} \langle x \rangle^{-\sigma}|  \leq M_2 C^{\beta-1} (\beta - 1)! \langle x \rangle^{1-\sigma - \beta}. $$
	For $\alpha \geq 1$
	\begin{align*}
	|\partial^{\alpha}_{\xi} \lambda_2(x, \xi)| &\leq M_2 h^{-\alpha} \left|w^{(\alpha)} \left(\frac{\xi}{h}\right)\right| \int^{\langle x \rangle}_{0} y^{-\sigma} dy \\
	&\leq \frac{M_2}{1-\sigma} C^{\alpha + 1}_{w} \langle R_{a_3} \rangle^{\alpha} \alpha!^{\mu} \chi_{E_{h, R_{a_3}}} (\xi) \langle \xi \rangle_{h}^{-\alpha} \langle x \rangle^{1-\sigma}. 
	\end{align*}
	Finally, for $\alpha, \beta \geq 1$
	\begin{align*}
	|\partial^{\alpha}_{\xi} \partial^{\beta}_{x} \lambda_2(x, \xi)| &\leq M_2 h^{-\alpha} \left|w^{(\alpha)} \left(\frac{\xi}{h}\right) \right| \partial^{\beta-1}_{x} \langle x \rangle^{-\sigma} \\
	&\leq M_2 C^{\alpha + 1}_{w} \langle R_{a_3} \rangle^{\alpha} C^{\beta - 1} \alpha!^{\mu} (\beta - 1)! \chi_{E_{h, R_{a_3}}} (\xi) \langle \xi \rangle_{h}^{-\alpha} \langle x \rangle^{1-\sigma - \beta}. 
	\end{align*}	
\end{proof}

For the function $\lambda_1$ we can prove the following alternative estimates.

\begin{lemma}\label{lemma_estimates_lambda_1}
	Let $\lambda_1(x,\xi)$ as in (\ref{equation_definition_of_lambda_1}). Then there exists $C>0$ such that for all $\alpha, \beta \geq 0$ and $(x,\xi) \in \R^2:$
	\begin{itemize}
		\item[(i)] $|\partial^{\alpha}_{\xi} \partial^{\beta}_{x} \lambda_1(x, \xi)| \leq M_1 C^{\alpha + \beta + 1} (\alpha! \beta!)^{\mu} \langle \xi \rangle^{-1 - \alpha}_{h} \langle x \rangle^{1-\frac{\sigma}{2} - \beta};$
		\item[(ii)] $|\partial^{\alpha}_{\xi} \partial^{\beta}_{x} \lambda_1(x, \xi)| \leq M_1 C^{\alpha + \beta + 1} (\alpha! \beta!)^{\mu} \langle \xi \rangle^{- \alpha}_{h} \langle x \rangle^{1-\sigma - \beta}.$
	\end{itemize}
	In particular $\lambda_1 \in \SG^{0, 1-\sigma}_{\mu}(\R^2)$.
\end{lemma}
\begin{proof}
	Denote by $\chi_{\xi}(x)$ the characteristic function of the set $\{x \in \R \colon \langle x \rangle^{\sigma} \leq \langle \xi \rangle^{2}_{h}\}$. For $\alpha = \beta = 0$ we have
	\begin{align*}
	|\lambda_1(x, \xi)| &\leq M_1 \left| w\left(\frac{\xi}{h}\right) \right| \langle \xi \rangle^{-1}_{h} \int^{\langle x \rangle}_{0} y^{-\frac{\sigma}2} dy \leq  \frac{2}{2-\sigma} M_1 \langle \xi \rangle^{-1}_{h} \langle x \rangle^{1-\frac{\sigma}2},
	\end{align*}
	and
	\begin{align*}
	|\lambda_1(x, \xi)| \leq M_1 \left| w\left(\frac{\xi}{h}\right) \right| \int^{\langle x \rangle}_{0} \langle \xi \rangle^{-1}_h  \langle y \rangle^{-\frac{\sigma}2} \chi_{\xi}(y) dy 
						\leq \frac{M_1}{1-\sigma} \langle x \rangle^{1-\sigma}.
	\end{align*}
	For $\alpha \geq 1$, with the aid of Fa\`a di Bruno formula, we have
	
	\begin{align*}
	|\partial^{\alpha}_{\xi}\lambda_1(x, \xi)| &\leq M_1 \sum_{\alpha_1 + \alpha_2 + \alpha_3 = \alpha} \dfrac{\alpha!}{\alpha_1!\alpha_2!\alpha_3!} h^{-\alpha_1} \left| w^{(\alpha_1)} \left( \frac{\xi}{h} \right) \right| \partial^{\alpha_2}_{\xi}\langle\xi\rangle^{-1}_{h} \left| \int_{0}^{x} \langle y \rangle^{-\frac{\sigma}2} \partial^{\alpha_3}_{\xi} \psi\left(\frac{\langle y \rangle^{\sigma}}{\langle \xi \rangle^{2}_{h}}\right) dy \right|  \\
	&\leq  M_1 \sum_{\alpha_1 + \alpha_2 + \alpha_3 = \alpha} \dfrac{\alpha!}{\alpha_1!\alpha_2!\alpha_3!} C^{\alpha_1+1}_{w} \langle R_{a_3} \rangle^{\alpha_3} \alpha_1!^{\mu} 
	\langle \xi \rangle^{-\alpha_1}_{h} C^{\alpha_2}\alpha_2! \langle \xi \rangle^{-1 - \alpha_2}_{h} \\
	&\times \int_{0}^{\langle x\rangle} \langle y \rangle^{-\frac{\sigma}2} \chi_{\xi}(y) \sum_{j = 1}^{\alpha_3} \frac{\left| \psi^{(j)} \left(\frac{\langle y \rangle^{\sigma}}{\langle \xi \rangle^{2}_{h}}\right) \right|}{j!} \sum_{\gamma_1 + \dots + \gamma_j = \alpha_3} \frac{\alpha_3!}{\gamma_1!\dots\gamma_j!} \prod_{\ell = 1}^{j} \partial^{\gamma_\ell}_{\xi} \langle \xi \rangle^{-2}_{h} \langle y \rangle^{\sigma} dy \\
	&\leq  M_1 \sum_{\alpha_1 + \alpha_2 + \alpha_3 = \alpha} \dfrac{\alpha!}{\alpha_1!\alpha_2!\alpha_3!} C^{\alpha_1+1}_{w} \langle R_{a_3} \rangle^{\alpha_3} \alpha_1!^{\mu} 
	\langle \xi \rangle^{-\alpha_1}_{h} C^{\alpha_2}\alpha_2! \langle \xi \rangle^{-1 - \alpha_2}_{h} \\
	&\times \int_{0}^{\langle x\rangle} \langle y \rangle^{-\frac{\sigma}2} \chi_{\xi}(y) \sum_{j = 1}^{\alpha_3} C^{j+1}_{\psi} j!^{\mu - 1} \sum_{\gamma_1 + \dots + \gamma_j = \alpha_3} \frac{\alpha_3!}{\gamma_1!\dots\gamma_j!} \prod_{\ell = 1}^{j} C^{\gamma_\ell + 1} \gamma_{\ell}! \langle \xi \rangle^{- \gamma_\ell}_{h} dy \\
	&\leq M_1 C_{\{w, \psi, \sigma, R_{a_3}\}}^{\alpha+1} \alpha!^{\mu} \langle \xi \rangle^{-1 - \alpha}_{h} \langle y \rangle^{1-\frac{\sigma}{2}},
	\end{align*}
	\noindent and
	\begin{align*}
	|\partial^{\alpha}_{\xi}\lambda_1(x, \xi)| &\leq M_1 \sum_{\alpha_1 + \alpha_2 + \alpha_3 = \alpha} \dfrac{\alpha!}{\alpha_1!\alpha_2!\alpha_3!} h^{-\alpha_1} \left| w^{(\alpha_1)} \left( \frac{\xi}{h} \right) \right| \partial^{\alpha_2}_{\xi}\langle\xi\rangle^{-1}_{h} \left| \int_{0}^{x} \langle y \rangle^{-\frac{\sigma}2} \partial^{\alpha_3}_{\xi} \psi\left(\frac{\langle y \rangle^{\sigma}}{\langle \xi \rangle^{2}_{h}}\right) dy \right|  \\
	&\leq  M_1 \sum_{\alpha_1 + \alpha_2 + \alpha_3 = \alpha} \dfrac{\alpha!}{\alpha_1!\alpha_2!\alpha_3!} C^{\alpha_1+1}_{w} \langle R_{a_3} \rangle^{\alpha_3} \alpha_1!^{\mu} 
	\langle \xi \rangle^{-\alpha_1}_{h} C^{\alpha_2}\alpha_2! \langle \xi \rangle^{- \alpha_2}_{h} \\
	&\times \int_{0}^{\langle x\rangle} \langle \xi \rangle^{-1}_{h} \langle y \rangle^{-\frac{\sigma}2} \chi_{\xi}(y) \sum_{j = 1}^{\alpha_3} \frac{\left| \psi^{(j)} \left(\frac{\langle y \rangle^{\sigma}}{\langle \xi \rangle^{2}_{h}}\right) \right|}{j!} \sum_{\gamma_1 + \dots + \gamma_j = \alpha_3} \frac{\alpha_3!}{\gamma_1!\dots\gamma_j!} \prod_{\ell = 1}^{j} \partial^{\gamma_\ell}_{\xi} \langle \xi \rangle^{-2}_{h} \langle y \rangle^{\sigma} dy \\
	&\leq  M_1 \sum_{\alpha_1 + \alpha_2 + \alpha_3 = \alpha} \dfrac{\alpha!}{\alpha_1!\alpha_2!\alpha_3!} C^{\alpha_1+1}_{w} \langle R_{a_3} \rangle^{\alpha_3} \alpha_1!^{\mu} 
	\langle \xi \rangle^{-\alpha_1}_{h} C^{\alpha_2}\alpha_2! \langle \xi \rangle^{- \alpha_2}_{h} \\
	&\times \int_{0}^{\langle x\rangle} \langle y \rangle^{-\sigma} \chi_{\xi}(y) \sum_{j = 1}^{\alpha_3} C^{j+1}_{\psi} j!^{\mu - 1} \sum_{\gamma_1 + \dots + \gamma_j = \alpha_3} \frac{\alpha_3!}{\gamma_1!\dots\gamma_j!} \prod_{\ell = 1}^{j} C^{\gamma_\ell + 1} \gamma_{\ell}! \langle \xi \rangle^{- \gamma_\ell}_{h} dy \\
	&\leq M_1 C_{\{w, \psi, \sigma, R_{a_3}\}}^{\alpha+1} \alpha!^{\mu} \langle \xi \rangle^{- \alpha}_{h} \langle y \rangle^{1-\sigma}.
	\end{align*}
	
	\noindent For $\beta \geq 1$ we have 
	
	\begin{align*}
	|\partial^{\beta}_{x} \lambda_1(x, \xi)| &\leq M_1 \langle \xi \rangle^{-1}_h \chi_{\xi} (x)\sum_{\beta_1 + \beta_2 = \beta - 1} \frac{(\beta-1)!}{\beta_1!\beta_2!} \partial^{\beta_1}_{x} \langle x \rangle^{-\frac{\sigma}2} \sum_{j=1}^{\beta_2} \frac{\left| \psi^{(j)} \left(\frac{\langle y \rangle^{\sigma}}{\langle \xi \rangle^{2}_{h}}\right) \right|}{j!}\\
	&\times  \sum_{\delta_1 + \dots + \delta_j = \beta_2}   \frac{\beta_2!}{\delta_1! \dots \delta_j!} \prod_{\ell = 1}^{j} \langle \xi \rangle^{-2}_{h}\partial^{\delta_\ell}_{x} \langle x \rangle^{\sigma} \\
	&\leq M_1 \langle \xi \rangle^{-1}_h \chi_{\xi} (x)\sum_{\beta_1 + \beta_2 = \beta - 1} \frac{(\beta-1)!}{\beta_1!\beta_2!} C^{\beta_1 + 1} 
	\beta_1!^{\mu} \langle x \rangle^{-\frac{\sigma}2 - \beta_1} \sum_{j=1}^{\beta_2} C_{\psi}^{j+1}j!^{\mu - 1}  \\
	&\times \sum_{\delta_1 + \dots +\delta_j = \beta_2} \frac{\beta_2!}{\delta_1! \dots \delta_j!} \prod_{\ell = 1}^{j} C^{\delta_\ell + 1} \delta_{\ell}! \langle x \rangle^{-\delta_\ell} \\  
	&\leq M_1 C_{\psi}^{\alpha + \beta + 1} (\beta - 1)!^{\mu} \langle \xi \rangle^{-1}_h \chi_{\xi} (x) \langle x \rangle^{1-{\sigma} - \beta} \\
	&\leq M_1 C_{\psi}^{\alpha + \beta + 1} (\beta - 1)!^{\mu} \langle x \rangle^{1-\sigma - \beta}.
	\end{align*} 
	
	\noindent Finally, for $\alpha, \beta \geq 1$ we have 
	
	\begin{align*}
	|\partial^{\alpha}_{\xi} &\partial^{\beta}_{x} \lambda_{1}(x, \xi)| \leq M_1 \sum_{\alpha_1 + \alpha_2 + \alpha_3 = \alpha} \dfrac{\alpha!}{\alpha_1!\alpha_2!\alpha_3!} h^{-\alpha_1} \left| w^{(\alpha_1)} \left( \frac{\xi}{h} \right) \right| \partial^{\alpha_2}_{\xi}\langle\xi\rangle^{-1}_{h} \sum_{\beta_1 + \beta_2 = \beta - 1} \frac{(\beta-1)!}{\beta_1!\beta_2!}  \\
	&\times \partial^{\beta_1}_{x} \langle x\rangle^{\frac{\sigma}2} \left| \partial^{\alpha_3}_{\xi} \partial^{\beta_2}_{x} \psi\left(\frac{\langle x \rangle^{\sigma}}{\langle \xi \rangle^{2}_{h}}\right) \right| \\
	&\leq M_1 \chi_{\xi}(x)\sum_{\alpha_1 + \alpha_2 + \alpha_3 = \alpha} \dfrac{\alpha!}{\alpha_1!\alpha_2!\alpha_3!} h^{-\alpha_1} \left| w^{(\alpha_1)} \left( \frac{\xi}{h} \right) \right| \partial^{\alpha_2}_{\xi}\langle\xi\rangle^{-1}_{h} \sum_{\beta_1 + \beta_2 = \beta - 1} \frac{(\beta-1)!}{\beta_1!\beta_2!} \partial^{\beta_1}_{x} \langle x\rangle^{-\frac{\sigma}2} \\
	&\times \sum_{j=1}^{\alpha_3 + \beta_2} \frac{\left| \psi^{(j)} \left(\frac{\langle x \rangle^{\sigma}}{\langle \xi \rangle^{2}_{h}}\right) \right|}{j!} 
	\sum_{\gamma_1 + \dots + \gamma_j = \alpha_3} \sum_{\delta_1 + \dots +\delta_j = \beta_2} \frac{\alpha_3!}{\gamma_1! \dots \gamma_j!} \frac{\beta_2!}{\delta_1! \dots \delta_j!}
	\prod_{\ell = 1}^{j} \partial^{\gamma_{\ell}}_{\xi} \langle \xi \rangle^{-2}_{h} \partial^{\delta{\ell}}_{x} \langle x \rangle^{\sigma} \\
	&\leq M_1 \chi_{\xi}(x)\sum_{\alpha_1 + \alpha_2 + \alpha_3 = \alpha} \dfrac{\alpha!}{\alpha_1!\alpha_2!\alpha_3!} C^{\alpha_1+1}_{w} \alpha_1!^{\mu} \langle R_{a_3} \rangle^{\alpha_1} \langle \xi \rangle^{-\alpha_1}_{h} C^{\alpha_2 + 1} \alpha_2! \langle\xi\rangle^{-1 - \alpha_2}_{h}  \\
	&\times \sum_{\beta_1 + \beta_2 = \beta - 1} \frac{(\beta-1)!}{\beta_1!\beta_2!} C^{\beta_1+1} \beta_1! \langle x\rangle^{-\frac12(1-\frac1s)-\beta_1} \sum_{j=1}^{\alpha_3 + \beta_2} C^{j+1}_{\psi} j!^{\mu-1}  \\
	&\times \sum_{\gamma_1 + \dots + \gamma_j = \alpha_3} \sum_{\delta_1 + \dots + \delta_j = \beta_2}\frac{\alpha_3!}{\gamma_1! \dots \gamma_j!} 
	\frac{\beta_2!}{\delta_1! \dots \delta_j!} \prod_{\ell = 1}^{j} C^{\gamma_\ell + 1} \gamma_\ell! \langle \xi \rangle^{-2 - \gamma_\ell}_{h} C^{\delta_\ell + 1} \delta_{\ell}! \langle x \rangle^{\sigma - \delta_{\ell}} \\
	&\leq M_1 \chi_{\xi}(x) C^{\alpha + \beta + 1}_{\{w, \sigma, \psi, R_{a_3}\}} \alpha!^{\mu} (\beta-1)!^{\mu} \langle \xi \rangle^{-1 - \alpha}_{h} 
	\langle x \rangle^{1-\frac{\sigma}{2} - \beta}  \\
	&\leq M_1 C^{\alpha + \beta + 1}_{\{w,\sigma, \psi, R_{a_3}\}} \alpha!^{\mu} (\beta-1)!^{\mu} \langle \xi \rangle^{- \alpha}_{h} 
	\langle x \rangle^{ 1-\sigma - \beta}.
	\end{align*}
\end{proof}


\subsection{Invertibility of $e^{\tilde{\Lambda}}$, $\tilde{\Lambda} = \lambda_2 + \lambda_1$}\label{section_invertibility_of_e_power_tilde_Lambda}
 
In this section we construct an inverse for the operator $e^{\tilde\Lambda}(x,D)$ with $\tilde{\Lambda}(x, \xi) = \lambda_2(x,\xi) + \lambda_1(x, \xi)$ and we prove that the inverse acts continuously on Gelfand Shilov-Sobolev spaces. By Lemmas \ref{lemma_estimates_lambda_2} and \ref{lemma_estimates_lambda_1} we have 
$\tilde{\Lambda} \in \SG^{0, 1-\sigma}_{\mu}(\R^{2})$. Therefore, by Proposition \ref{proposition_exponential_of_symbols_of_finite_order}, $e^{\tilde{\Lambda}} \in \SG^{0, \infty}_{\mu; 1/(1-\sigma)}(\R^{2})$. To construct the inverse of $e^{\tilde{\Lambda}}(x,D)$ we need to use the  $L^2$ adjoint of $e^{-\tilde{\Lambda}}(x,D)$, denoted in the sequel by $^R \hskip-1pt e^{ -\tilde\Lambda}$ and defined as an oscillatory integral by
$$
^R \hskip-1pt e^{ -\tilde\Lambda}u(x)= \iint e^{i(x-y)\xi} e^{ -\tilde\Lambda(y,\xi)} u(y)\, dy \dslash \xi.
$$
Assuming $\mu > 1$ such that $1/(1-\sigma) > 2\mu -1$, by results from calculus, we may write
$$^{R}e^{-\tilde{\Lambda}} = a_1(x,D) + r_1(x,D),$$
where $a_1 \sim \sum_{\alpha} \frac{1}{\alpha!} \partial^{\alpha}_{\xi}D^{\alpha}_{x}e^{-\tilde{\Lambda}}$ in $FSG^{0, \infty}_{\mu; 1/(1-\sigma)}(\R^{2})$, $r_1 \in \mathcal{S}_{2\mu - 1}(\R^{2})$, and
$$
e^{\tilde{\Lambda}} \circ ^{R}e^{-\tilde{\Lambda}} = e^{\tilde{\Lambda}}\circ a_1(x,D) + e^{\tilde{\Lambda}}\circ r_1(x,D) = a_2(x,D) + r_2(x,D) + e^{\tilde{\Lambda}}\circ r_1(x,D),
$$
where 
$$
a_2 \sim \sum_{\alpha, \beta} \frac{1}{\alpha! \beta!} \partial^{\alpha}_{\xi}e^{\tilde{\Lambda}}\partial^{\beta}_{\xi}D^{\alpha+\beta}_{x}e^{-\tilde{\Lambda}} = \sum_{\gamma} \frac{1}{\gamma!} \partial^{\gamma}_{\xi}(e^{\tilde{\Lambda}}D^{\gamma}_{x}e^{-\tilde{\Lambda}}) \,\,\text{in}\,\, FSG^{0, \infty}_{\mu; 1/(1-\sigma)}(\R^2) 
$$
and $r_2 \in \mathcal{S}_{2\mu -1}(\R^{2})$. Therefore 
$$
e^{\tilde{\Lambda}} \circ ^{R}e^{-\tilde{\Lambda}} = a(x,D) + r(x,D),
$$
where $a \sim \sum_{\gamma} \frac{1}{\gamma!} \partial^{\gamma}_{\xi}(e^{\tilde{\Lambda}}D^{\gamma}_{x}e^{-\tilde{\Lambda}})$ in $FSG^{0, \infty}_{\mu; 1/(1-\sigma)}(\R^{2})$ and $r \in \mathcal{S}_{2\mu - 1}(\R^{2})$.

Now let us study more carefully the asymptotic expansion 
$$
\sum_{\gamma \geq 0} \frac{1}{\gamma!} \partial^{\gamma}_{\xi}(e^{\tilde{\Lambda}}D^{\gamma}_{x}e^{-\tilde{\Lambda}}) 
= \sum_{\gamma \geq 0} r_{1, \gamma}.
$$
Note that
\begin{align*}
	e^{\tilde{\Lambda}(x,\xi)} D^{\gamma}_{x}e^{-\tilde{\Lambda}(x,\xi)} &= \sum_{j = 1}^{\gamma} \frac{(-1)^{\gamma}}{j!} \sum_{\gamma_1 + \dots + \gamma_j = \gamma} 
	\frac{{\gamma!}}{\gamma_1! \dots \gamma_j!} \prod_{\ell = 1}^{j} D^{\gamma_\ell}_{x} \tilde{\Lambda}(x, \xi),
\end{align*}

\noindent hence, for $\alpha, \beta \geq 0$,

\begin{align*}
	|\partial^{\alpha}_{\xi} \partial^{\beta}_{x} r_{1, \gamma}| &\leq \frac{1}{\gamma!} \sum_{j = 1}^{\gamma} \frac{1}{j!} \sum_{\gamma_1 + \dots + \gamma_j = \gamma} \frac{{\gamma!}}{\gamma_1! \dots \gamma_j!} 
	\sum_{\alpha_1 + \dots + \alpha_j = \alpha + \gamma} \sum_{\beta_1 + \dots + \beta_j = \beta} 
	\frac{(\alpha+\gamma)!}{\alpha_1! \dots \alpha_j!} \frac{\beta!}{\beta_1!\dots \beta_j!} \\
	&\times\prod_{\ell = 1}^{j} |\partial^{\alpha_\ell}_{\xi}\partial^{\beta_\ell + \gamma_\ell}_{x} \tilde{\Lambda}(x, \xi)| \\
	&\leq \frac{1}{\gamma!} \sum_{j = 1}^{\gamma} \frac{1}{j!} \sum_{\gamma_1 + \dots + \gamma_j = \gamma} \frac{{\gamma!}}{\gamma_1! \dots \gamma_j!} 
	\sum_{\alpha_1 + \dots + \alpha_j = \alpha + \gamma} \sum_{\beta_1 + \dots + \beta_j = \beta} 
	\frac{(\alpha+\gamma)!}{\alpha_1! \dots \alpha_j!} \frac{\beta!}{\beta_1!\dots \beta_j!} \\
	&\times \prod_{\ell = 1}^{j}C_{\tilde{\Lambda}}^{\alpha_\ell + \beta_\ell + \gamma_\ell + 1} \alpha_{\ell}!^{\mu}(\beta_\ell + \gamma_{\ell})!^{\mu} 
	\langle \xi \rangle_{h}^{-\alpha_\ell}\langle x \rangle^{1-\sigma - \beta_\ell - \gamma_\ell} \\
	&\leq C^{\alpha + \beta + 2\gamma + 1} \alpha!^{\mu} \beta!^{\mu} \gamma!^{2\mu - 1} \langle \xi \rangle_{h}^{-\alpha - \gamma} \sum_{j = 1}^{\gamma} \frac{\langle x \rangle^{(1-\sigma)j - \beta - \gamma}}{j!}.
\end{align*}

In the following we shall consider the sets 
$$
Q_{t_1, t_2; h} = \{(x, \xi) \in \R^{2} : \langle x \rangle < t_1 \,\, \text{and} \,\, \langle \xi \rangle_{h} < t_2 \}
$$
and $Q_{t_1, t_2; h}^{e} = \R^{2} \setminus Q_{t_1, t_2; h}$. When $t_1 = t_2 = t$ we simply write $Q_{t; h}$ and $Q^{e}_{t; h}$.

Let $\psi(x, \xi) \in C^{\infty}(\R^{2})$ such that $\psi \equiv 0$ on $Q_{2; h}$, $\psi \equiv 1$ on $Q^{e}_{3; h}$, $0 \leq \psi \leq 1$ and 
$$
|\partial^{\alpha}_{\xi} \partial^{\beta}_{x} \psi (x, \xi)| \leq C^{\alpha + \beta + 1}_{\psi}\alpha!^{\mu} \beta!^{\mu},
$$
for every $x, \xi \in \R$ and $\alpha, \beta \in \N_0$. Now set $\psi_0 \equiv 1$ and, for $j \geq 1$, 
$$
\psi_j (x, \xi) := \psi \left( \dfrac{x}{R(j)}, \dfrac{\xi}{R(j)} \right),
$$
where $R(j) := R j^{2\mu-1}$ and $R > 0$ is a large constant. Let us recall that 
\begin{itemize}
	\item $(x, \xi) \in Q^{e}_{3R(j)} \implies \left(  \dfrac{x}{R(j)}, \dfrac{\xi}{R(j)}  \right) \in Q^{e}_{3} \implies \psi_i(x, \xi) = 1$, for $i \leq j$;  
	\item $(x, \xi) \in Q_{R(j)} \implies \left(  \dfrac{x}{R(j)}, \dfrac{\xi}{R(j)}  \right) \in Q_{2} \implies 
	\psi_i(x, \xi) = 0$, for $i \geq j$.  
\end{itemize} 

Defining $b(x,\xi) = \sum_{j \geq 0} \psi_{j}(x,\xi) r_{1, j}(x,\xi)$
we have that $b  \in \SG^{0,\infty}_{\mu;\frac{1}{1-\sigma}}(\R^{2})$ and
$$
b(x,\xi) \sim \sum_{j \geq 0} r_{1, j}(x,\xi) \,\, \text{in} \,\, FSG^{0, \infty}_{\mu;\frac{1}{1-\sigma}}(\R^{2}). 
$$
We will show that $b \in SG^{0,0}_{\mu}(\R^{2n})$. Indeed, first we write 
$$
b(x,\xi) = 1 + \sum_{j \geq 1} \psi_{j}(x,\xi) r_{1,j}(x,\xi) = 1 + \sum_{j \geq 0} \psi_{j+1}(x,\xi) r_{1, j+1}(x,\xi).
$$
On the support of  $\partial^{\alpha_1}_{\xi}\partial^{\beta_1}_{x} \psi_{j+1}$ we have
$$
\langle x \rangle \leq 3R(j+1) \quad \text{and} \quad \langle \xi \rangle_{h} \leq 3 R(j+1),
$$
whenever $\alpha_1 + \beta_1 \geq 1$. Hence
\begin{align*}
	\lvert \partial^{\alpha}_{\xi} \partial^{\beta}_{x} & \sum_{j \geq 0} \psi_{j+1} r_{1, j+1} (x,\xi) \rvert  \leq 
	\sum_{\overset{\alpha_1+\alpha_2 = \alpha}{\beta_1+\beta_2 = \beta}} \frac{\alpha!}{\alpha_1!\alpha_2!} \frac{\beta!}{\beta_1!\beta_2!}
	|\partial^{\alpha_1}_{\xi}\partial^{\beta_1}_{x} \psi_{j+1}(x,\xi)| |\partial^{\alpha_2}_{\xi} \partial^{\beta_2}_{x} r_{1, j+1}(x,\xi)| \\
	&\leq \sum_{j \geq 0} \sum_{\overset{\alpha_1+\alpha_2 = \alpha}{\beta_1+\beta_2 = \beta}} \frac{\alpha!}{\alpha_1!\alpha_2!} \frac{\beta!}{\beta_1!\beta_2!} \frac{1}{R(j+1)^{(\alpha_1 + \beta_1)}} C_{\psi}^{\alpha_1+\beta_1 +1} \alpha_1!^{\mu}\beta_1!^{\mu} \\
	&\times
	C^{\alpha_2+\beta_2 + 2(j+1) +1 } \alpha_2!^{\mu} \beta_2!^{\mu} (j+1)!^{2\mu-1} \langle \xi \rangle^{-\alpha_2-(j+1)}_{h} \langle x \rangle^{-\beta_1-(j+1)} \sum_{\ell = 1}^{j+1} \frac{\langle x \rangle^{(1-\sigma)\ell}}{\ell!} \\
	&\leq \sum_{j \geq 0} \sum_{\overset{\alpha_1+\alpha_2 = \alpha}{\beta_1+\beta_2 = \beta}} \frac{\alpha!}{\alpha_1!\alpha_2!} \frac{\beta!}{\beta_1!\beta_2!} \frac{1}{R(j+1)^{(\alpha_1+\beta_1)}} C_{\psi}^{\alpha_1+\beta_1 +1} \alpha_1!^{\mu}\beta_1!^{\mu} \\
	&\times
	C^{\alpha_2+\beta_2 + 2(j+1) +1 } \alpha_2!^{\mu} \beta_2!^{\mu} (j+1)!^{2\mu-1} \langle \xi \rangle^{-\alpha_2-(j+1)}_{h} \langle x \rangle^{-\sigma - \beta_1} \sum_{\ell = 1}^{j+1} \frac{\langle x \rangle^{(1-\sigma)(\ell-1) - j}}{\ell!} \\
	&\leq \tilde{C}^{\alpha+\beta+1}(\alpha!\beta!)^{\mu} \langle \xi \rangle^{-1 - \alpha}_{h} \langle x \rangle^{-\sigma - \beta} 
	\sum_{j \geq 0} C^{2j} (j+1)!^{2\mu-1} \langle \xi \rangle^{-j}_{h} \sum_{\ell = 0}^{j+1} \frac{\langle x \rangle^{(1-\sigma)(\ell-1) - j}}{\ell!}.
\end{align*}
We also have that 
$$
\langle x \rangle \geq R(j+1) \quad \text{or} \quad \langle \xi \rangle_{h} \geq R(j+1)
$$
holds true on the support of $\partial^{\alpha_1}_{\xi}\partial^{\beta_1}_{x} \psi_{j+1}$. If $\langle \xi \rangle_{h} \geq R(j+1)$, then 
$$
\langle \xi \rangle^{-j}_{h} \leq R^{-j}(j+1)^{-j(2\mu-1)} \leq R^{-j} (j+1)!^{-(2\mu-1)}.
$$ 
On the other hand, since we are assuming $\mu > 1$ such that $2\mu-1 < \frac{1}{1-\sigma}$, if $\langle x \rangle \leq R(j+1)$ we obtain
\begin{align*}
	\langle x \rangle^{(1-\sigma)(\ell-1) - j}	&\leq R^{(1-\sigma)(\ell-1) - j}\{(j+1)^{2\mu-1} \}^{(1-\sigma)(\ell-1) - j}  \\
	&\leq R^{-\sigma j} (j+1)^{\ell - 1 - j(2\mu-1)} \\
	&=  R^{-\sigma j} (j+1)^{\ell - 1} (j+1)!^{-(2\mu-1)}.
\end{align*}
Enlarging $R >0$ if necessary, we can infer that $\sum_{j \geq 1} r_{1, j} \in \SG^{-1, -\sigma}_{\mu}(\R^{2})$. 

In analogous way it is possible to prove that $\sum_{j \geq k} r_{1, j} \in \SG^{-k, -\sigma k}_{\mu}(\R^{2})$. Hence, we may conclude 
$$
b(x,\xi) - \sum_{j < k}r_{1,j}(x,\xi) \in \SG^{-k, -\sigma k}_{\mu}(\R^{2}), \quad k \in \N,
$$
that is, $b \sim \sum_{j} r_{1,j}$ in $\SG^{0,0}_{\mu}(\R^{2})$.

Since $a \sim \sum r_{1,j}$ in $FSG^{0, \infty}_{\mu;1/(1-\sigma)}(\R^{2})$, $b \sim \sum r_{1,j}$ in $FSG^{0,\infty}_{\mu;1/(1-\sigma)}(\R^{2})$ we have $a - b \in \mathcal{S}_{2\mu-1}(\R^{2})$. Thus we may write 
$$
e^{\tilde{\Lambda}}(x,D) \circ ^{R}e^{-\tilde{\Lambda}} = I + \tilde{r}(x,D) + \bar{r}(x,D) =I+r(x,D),
$$
where $\tilde{r} \in \SG^{-1, -\sigma}_{\mu}(\R^{2})$, $\tilde{r} \sim \sum_{\gamma \geq 1} r_{1, \gamma}$ in $\SG^{-1, -\sigma}_{\mu}(\R^{2})$ and $\bar{r} \in \mathcal{S}_{2\mu -1}(\R^{2})$. In particular $r \in \SG^{-1, -\sigma}_{2\mu-1}(\R^2)$, therefore we obtain
\begin{align*}
	|\partial^{\alpha}_{\xi} \partial^{\beta}_{x} r(x,\xi)| &\leq C_{\alpha, \beta} \langle \xi \rangle^{-1 -\alpha}_{h} \langle x \rangle^{-\sigma -\beta}  \\
	&\leq C_{\alpha, \beta} h^{-1} \langle \xi \rangle^{-\alpha}_{h} \langle x \rangle^{-\sigma - \beta}.
\end{align*}
This implies that the $(0,0)-$seminorms of $r$ are bounded by $h^{-1}$. Choosing $h$ large enough, we obtain that $I + r(x,D)$ is invertible on $L^{2}(\R)$ and its inverse $(I+r(x,D))^{-1}$ is given by the Neumann series $\sum_{j \geq 0}(-r(x,D))^{j}$. 

By Theorem \ref{theorem_spectral_invariance_SG_operator_with_gevrey_estimates} we have
$$
(I+r(x,D))^{-1}= q(x,D) + k(x,D),
$$
where $q \in \SG^{0,0}_{2\mu -1}(\R^{2})$, $k \in \Sigma_{\delta}(\R^{2})$ for every $\delta > 2(2\mu-1) - 1  = 4\mu - 3 $. Choosing $\mu>1$ close enough to $1$, we have that $\delta$ can be chosen arbitrarily close to $1$. Hence, by Theorem \ref{theorem_continuity_finite_order_in_gelfand_shilov_sobolev_spaces}, for every fixed $s >1, \theta >1$, we can find $\mu >1$ such that 
$$
(I+r(x,D))^{-1} : H^{m'}_{\rho'; s, \theta}(\R) \to H^{m'}_{\rho'; s, \theta}(\R)
$$
is continuous for every $m', \rho' \in \R^2$. Analogously one can show the existence of a left inverse of $e^\Lambda$ with the same properties. Summing up, we obtain the following result.

\begin{lemma}\label{lemma_inverse_of_e_power_tilde_Lambda}
	 Let $s, \theta > 1$ and take $\mu > 1$ such that $\min\{s, \theta\} >  4\mu-3$. For $h > 0$ large enough, the operator $e^{\tilde{\Lambda}}(x,D)$ is invertible on $L^2(\R)$ and on $\Sigma_{\min\{s,\theta\}}(\R)$ and its inverse is given by 
	$$
	\{e^{\tilde{\Lambda}}(x,D)\}^{-1} = ^{R}e^{-\tilde{\Lambda}(x,D)} \circ (I+r(x,D))^{-1}=^{R}e^{-\tilde{\Lambda}(x,D)} \circ \sum_{j \geq 0} (-r(x,D))^{j}, 
	$$
	where $r \in \SG^{-1, -\sigma}_{2\mu - 1}(\R^{2})$ and $r \sim \sum_{\gamma \geq 1} \frac{1}{\gamma!} \partial^{\gamma}_{\xi}(e^{\tilde{\Lambda}} D^{\gamma}_{x} e^{-\tilde{\Lambda}})$ in $\SG^{-1, -\sigma}_{2\mu-1}(\R^{2})$. Moreover, the symbol of $(I+r(x,D))^{-1}$ belongs to $\SG^{0,0}_{\delta}(\R^2)$ for every $\delta >4\mu-3$ and it maps continuously $H^{m'}_{\rho';s,\theta}(\R)$ into itself for any $\rho', m' \in \R^2$.  
\end{lemma}
We conclude this section writing $\{e^{\tilde{\Lambda}}(x,D)\}^{-1}$ in a more precise way. From the asymptotic expansion of the symbol $r(x,\xi)$ we may write
$$
\{e^{\tilde{\Lambda}}(x,D)\}^{-1} = ^{R}e^{-\tilde{\Lambda}} \circ (I - r(x,D) + (r(x,D))^2 + q_{-3}(x,D)),
$$
where $q_{-3}$ denotes an operator with symbol in $\SG^{-3, -3\sigma}_{\delta}(\R^{2})$ for every $\delta >4\mu-3$. Now note that
$$
r = i\partial_{\xi} \partial_{x}\tilde{\Lambda} + \frac{1}{2}\partial^{2}_{\xi}(\partial^{2}_{x}\tilde{\Lambda} - [\partial_x \tilde{\Lambda}]^{2}) + q_{-3} 
= q_{-1} + q_{-2} + q_{-3}.
$$
Hence
\begin{eqnarray*}
(r(x,D))^2 &=& (q_{-1} + q_{-2} + q_{-3})(x,D) \circ (q_{-1} + q_{-2} + q_{-3})(x,D) \\ &=& q_{-1}(x,D) \circ q_{-1}(x,D) + q_{-3}(x,D) \\ &=&\textrm{op}\left\{ -[\partial_{\xi} \partial_{x} \tilde{\Lambda} ]^{2} + q_{-3} \right\}
\end{eqnarray*}
for a new element $q_{-3}$ in the same space.
We finally obtain:
\begin{equation}\label{equation_inverse_of_e_power_tilde_Lambda_in_a_precise_way}
\{e^{\tilde{\Lambda}}(x,D)\}^{-1} = ^{R}e^{-\tilde{\Lambda}} \circ \left[ I +\textrm{op}\left(- i\partial_{\xi} \partial_{x} \tilde{\Lambda} - \frac{1}{2}\partial^{2}_{\xi}(\partial^{2}_{x}\tilde{\Lambda} - [\partial_x \tilde{\Lambda}]^{2}) - [\partial_{\xi} \partial_{x} \tilde{\Lambda} ]^{2} + q_{-3} \right) \right],
\end{equation}
where $q_{-3} \in \SG^{-3, -3\sigma}_{\delta}(\R^{2})$.  Since we deal with operators whose order does not exceed 3, in the next sections we are going to use frequently formula \eqref{equation_inverse_of_e_power_tilde_Lambda_in_a_precise_way} for the inverse of $e^{\tilde{\Lambda}}(x,D)$. 

\section{Conjugation of $iP$}\label{section_conjugation_of_iP}

In this section we will perform the conjugation of $iP$ by the operator $e^{ \rho_1 \langle D \rangle^{\frac{1}{\theta}} } \circ e^{\Lambda(t, x, D) }$ and its inverse,
where $\Lambda(t,x,\xi)=k(t)\langle x\rangle^{1-\sigma}_h+\tilde\Lambda(x,\xi)$ and $k\in C^1([0,T];\R)$ is a non increasing function such that $k(T)\geq 0$. Since the arguments in the following involve also derivatives with respect to $t$ these derivatives will be denoted by $D_t$, whereas the symbol $D$ in the notation for pseudodifferential operators will always correspond to derivatives with respect to $x$.

More precisely, we will compute 
$$
 e^{\rho_1 \langle D \rangle^{\frac{1}{\theta}}}
\circ e^{k(t)\langle x \rangle^{1-\sigma}_{h}} \circ e^{\tilde{\Lambda}}(x,D) \circ (iP)(t,  x, D_t, D_x) \circ
\{e^{\tilde{\Lambda}}(x,D)\}^{-1} \circ e^{-k(t)\langle x \rangle^{1-\sigma}_{h}} \circ 
e^{-\rho_1 \langle D \rangle^{\frac{1}{\theta}}} ,
$$
where $\rho_1 \in \R$ and $P(t,  x, D_t, D_x)$ is given by \eqref{equation_main_class_of_3_evolution_of_pseudo_diff_operators}. As we discussed before, the role of this conjugation is to make positive the lower order terms of the conjugated operator. 

Since the operator $^{R}e^{-\tilde{\Lambda}}$ appears in the inverse $\{e^{\tilde{\Lambda}}(x,D)\}^{-1}$, we need the following technical lemma.

\begin{lemma}\label{lemma_conjugation_by_the_reverse_operator}
Let $\tilde{\Lambda} \in \SG^{0, 1-\sigma}_{\mu}(\R^2)$ and $a \in \SG^{m_1, m_2}_{1, s_0}(\R^2)$, with $\mu > 1$ such that $1/(1-\sigma) > \mu + s_0 - 1$ and $s_0 > \mu$. Then, for $M \in \N$,
$$
e^{\tilde{\Lambda}}(x,D) \circ a(x,D) \circ ^{R}e^{-\tilde{\Lambda}} = a(x,D) + \textrm{op}\left(\sum_{1 \leq \alpha + \beta < M} \frac{1}{\alpha! \beta!} \partial^{\alpha}_{\xi} \{ \partial^{\beta}_{\xi} e^{\tilde{\Lambda}} D^{\beta}_{x} a D^{\alpha}_{x} e^{-\tilde{\Lambda}}\} + q_{M} \right) + r(x,D),
$$
where $q_M \in \SG^{m_1 - M, m_2 - M\sigma}_{\mu, s_0}(\R^2)$ and $r \in \mathcal{S}_{\mu+s_0 -1}(\R^{2})$.
\end{lemma}
\begin{proof}
Since $e^{\pm\tilde{\Lambda}} \in \SG^{0, \infty}_{\mu; \frac{1}{1-\sigma}} (\R^{2})$ and $a \in \SG^{m_1, m_2}_{1, s_0} (\R^{2})$, we have $e^{\pm \tilde{\Lambda}}, a \in \SG^{0, \infty}_{\mu, s_0; \frac{1}{1-\sigma}} (\R^{2})$. Therefore, by results from calculus, we obtain
$$
^{R}e^{-\tilde{\Lambda}} = a_1(x,D) + r_1(x,D) \quad {\rm and}\quad e^{\tilde{\Lambda}} (x,D) \circ a(x,D) = a_2(x,D) + r_2(x,D),
$$
where $a_1 \in \SG^{0, \infty}_{\mu, s_0; \frac{1}{1-\sigma}}(\R^{2})$, $a_2 \in \SG^{m_1, \infty}_{\mu, s_0; \frac{1}{1-\sigma}}(\R^{2}),$ $r_1, r_2 \in \mathcal{S}_{\mu+s_0-1}(\R^{2})$ and 
$$
a_1 \sim \sum_{\alpha} \frac{1}{\alpha!} \partial^{\alpha}_{\xi} D^{\alpha}_{x} e^{-\tilde{\Lambda}} \,\, \text{in} \,\,
FSG^{0, \infty}_{\mu, s_0; \frac{1}{1-\sigma}}(\R^{2}),
$$
$$
a_2 \sim \sum_{\beta} \frac{1}{\beta!} \partial^{\beta}_{\xi}e^{\tilde{\Lambda}} D^{\beta}_{x} a \,\, \text{in} \,\,
FSG^{m_1, \infty}_{\mu, s_0; \frac{1}{1-\sigma}}(\R^{2}).
$$
Hence 
\begin{eqnarray}\nonumber
	e^{\tilde{\Lambda}} \circ a(x,D) \circ ^{R}e^{-\tilde{\Lambda}} &=& a_2(x,D) \circ a_1(x,D) + a_2(x,D) \circ r_1(x,D)
	\\ \nonumber
	&+& r_2(x,D) \circ a_1(x,D) + r_2(x,D) \circ r_1(x,D)
	\\ \nonumber
	&=&a_3(x,D) + r_3(x,D)+ a_2(x,D) \circ r_1(x,D) 
	\\\nonumber
	&+& r_2(x,D) \circ a_1(x,D) + r_2(x,D) \circ r_1(x,D),
\end{eqnarray} 
with $a_2(x,D) \circ a_1(x,D) = a_3(x,D) + r_3(x,D)$, where $a_3 \in \SG^{m_1, \infty}_{\mu, s_0; \frac{1}{1-\sigma}}(\R^{2})$, $r_3 \in \mathcal{S}_{\mu+s_0-1}(\R^{2})$ and 
\begin{align*}
	a_3 &\sim \sum_{\gamma, \alpha, \beta} \frac{1}{\alpha!\beta!\gamma!} \, \partial^{\gamma}_{\xi} \{\partial^{\beta}_{\xi} e^{\tilde{\Lambda}} D^{\beta}_{x}a\} \partial^{\alpha}_{\xi}D^{\alpha + \gamma}_{x}e^{-\tilde{\Lambda}} = \sum_{\alpha, \beta} \frac{1}{\alpha! \beta!} \, \partial^{\alpha}_{\xi} \{ \partial^{\beta}_{\xi} e^{\tilde{\Lambda}} D^{\beta}_{x} a D^{\alpha}_{x} e^{-\tilde{\Lambda}}\}
	\\
	&=\ds\sum_{j\geq 0}\sum_{\alpha+\beta=j}\frac{1}{\alpha! \beta!} \, \partial^{\alpha}_{\xi} \{ \partial^{\beta}_{\xi} e^{\tilde{\Lambda}} D^{\beta}_{x} a D^{\alpha}_{x} e^{-\tilde{\Lambda}}\}=:\sum_{j\geq 0} c_j \,\, \text{in} \,\, FSG^{m_1, \infty}_{\mu, s_0; \frac{1}{1-\sigma}}.
\end{align*}
Thus
$$
e^{\tilde{\Lambda}}(x,D) \circ a(x,D) \circ ^{R}e^{-\tilde{\Lambda}} = a_3(x,D) + r(x,D), 
$$
for some $r \in \mathcal{S}_{\mu+s_0-1}(\R^{2})$.

Now let us study the asymptotic expansion of $a_3$. For $\alpha, \beta \in \N_0$ we have (omitting the dependence $(x,\xi)$):
\begin{align*}
	\partial^{\beta}_{\xi} e^{\tilde{\Lambda}} \partial^{\beta}_{x} a \partial^{\alpha}_{x} e^{-\tilde{\Lambda}} &= \partial^{\beta}_{x} a \sum_{h = 1}^{\beta} \frac{1}{h!} 
	\sum_{\beta_1 + \dots + \beta_h = \beta} \frac{\beta!}{\beta_1! \dots \beta_h!} \prod_{\ell = 1}^{h} \partial^{\beta_\ell}_{\xi} \tilde{\Lambda}  \\
	&\times \sum_{k = 1}^{\alpha} \frac{1}{k!} 
	\sum_{\alpha_1 + \dots + \alpha_k = \alpha} \frac{\alpha!}{\alpha_1! \dots \alpha_k!} \prod_{\ell = 1}^{k} \partial^{\alpha_\ell}_{x} (-\tilde{\Lambda}).
\end{align*}
Therefore, by Fa\`a di Bruno formula, for $\gamma, \delta \in \N_0$, we have
\begin{align*}
	\partial^{\gamma + \alpha}_{\xi} \partial^{\delta}_{x} &\{\partial^{\beta}_{\xi} e^{\tilde{\Lambda}} \partial^{\beta}_{x} a \partial^{\alpha}_{x} e^{-\tilde{\Lambda}}\} = 
	\sum_{\gamma_1 + \gamma_2 + \gamma_3 = \gamma + \alpha} \sum_{\delta_1 + \delta_2 + \delta_3 = \delta} \frac{(\gamma + \alpha)!}{\gamma_1! \gamma_2! \gamma_3!}
	\frac{\delta!}{\delta_1! \delta_2! \delta_3!} \,\, \partial^{\gamma_1}_{\xi} \partial^{\beta + \delta_1}_{x} a \\
	&\times \partial^{\gamma_2}_{\xi} \partial^{\delta_2}_{x} 
	\left( 
	\sum_{h = 1}^{\beta} \frac{1}{h!}\sum_{\beta_1 + \dots + \beta_h = \beta} \frac{\beta!}{\beta_1! \dots \beta_h!} \prod_{\ell = 1}^{h} 
	\partial^{\beta_\ell}_{\xi} \tilde{\Lambda}
	\right)  \\
	&\times \partial^{\gamma_3}_{\xi} \partial^{\delta_3}_{x}
	\left(
	\sum_{k = 1}^{\alpha} \frac{1}{k!} \sum_{\alpha_1 + \dots + \alpha_k = \alpha} \frac{\alpha!}{\alpha_1! \dots \alpha_k!} 
	\prod_{\ell = 1}^{k} \partial^{\alpha_\ell}_{x} (-\tilde{\Lambda})
	\right)  \\
	&= \sum_{\gamma_1 + \gamma_2 + \gamma_3 = \gamma + \alpha} \sum_{\delta_1 + \delta_2 + \delta_3 = \delta} \frac{(\gamma + \alpha)!}{\gamma_1! \gamma_2! \gamma_3!}
	\frac{\delta!}{\delta_1! \delta_2! \delta_3!} \,\, \partial^{\gamma_1}_{\xi} \partial^{\beta + \delta_1}_{x} a \\
	&\times  
	\sum_{h = 1}^{\beta} \frac{1}{h!}\sum_{\beta_1 + \dots + \beta_h = \beta} \frac{\beta!}{\beta_1! \dots \beta_h!} \sum_{\theta_1 + \dots + \theta_h = \gamma_2} \sum_{\sigma_1 + \dots + \sigma_h = \delta_2} \frac{\gamma_2!}{\theta_1! \dots \theta_h!} \frac{\delta_2!}{\sigma_1! \dots \sigma_h!} \\
	&\times \prod_{\ell = 1}^{h} \partial^{\theta_\ell + \beta_\ell}_{\xi} \partial^{\sigma_\ell}_{x} \tilde{\Lambda} \\
	&\times  
	\sum_{k = 1}^{\alpha} \frac{1}{k!}\sum_{\alpha_1 + \dots + \alpha_k = \alpha} \frac{\alpha!}{\alpha_1! \dots \alpha_k!} \sum_{\theta_1 + \dots + \theta_k = \gamma_3} \sum_{\sigma_1 + \dots + \sigma_k = \delta_3} \frac{\gamma_3!}{\theta_1! \dots \theta_k!} \frac{\delta_3!}{\sigma_1! \dots \sigma_k!} \\
	&\times \prod_{\ell = 1}^{k} \partial^{\theta_\ell}_{\xi} \partial^{\alpha_\ell + \sigma_\ell}_{x} (-\tilde{\Lambda}),
\end{align*}
hence
\begin{align*}
	|\partial^{\gamma+\alpha}_{\xi} \partial^{\delta}_{x}(\partial^{\beta}_{\xi} e^{\tilde{\Lambda}} &D^{\beta}_{x}a D^{\alpha}_{x}e^{-\tilde{\Lambda}})| \leq
	\sum_{\overset{\gamma_1 + \gamma_2 + \gamma_3 = \gamma + \alpha}{\delta_1 + \delta_2 + \delta_3 = \delta}} 
	\frac{(\gamma + \alpha)!}{\gamma_1! \gamma_2! \gamma_3!} \frac{\delta!}{\delta_1! \delta_2! \delta_3!} C^{\gamma_1 + \beta + \delta_1 + 1}_{a} \gamma_1!^{\mu}(\beta+\gamma_1)!^{s_0} \\ 
	&\times  \langle \xi \rangle^{m_1 - \gamma_1} \langle x \rangle^{m_2 - \beta - \delta_1}  \\
	&\times  
	\sum_{h= 1}^{\beta} \frac{1}{h!}\sum_{\beta_1 + \dots + \beta_h = \beta} \frac{\beta!}{\beta_1! \dots \beta_h!} \sum_{\theta_1 + \dots + \theta_h = \gamma_2} \sum_{\sigma_1 + \dots + \sigma_h = \delta_2} \frac{\gamma_2!}{\theta_1! \dots \theta_h!} \frac{\delta_2!}{\sigma_1! \dots \sigma_h!} \\
	&\times \prod_{\ell = 1}^{h} C^{\theta_\ell + \beta_\ell + \sigma_\ell + 1}_{\tilde{\Lambda}} (\theta_\ell + \beta_\ell)!^{\mu}\sigma_{\ell}!^{\mu} 
	\langle \xi \rangle^{-\theta_\ell - \beta_\ell} \langle x \rangle^{1-\sigma - \sigma_{\ell}}  \\
	&\times  
	\sum_{k = 1}^{\alpha} \frac{1}{k!}\sum_{\alpha_1 + \dots + \alpha_k = \alpha} \frac{\alpha!}{\alpha_1! \dots \alpha_k!} \sum_{\theta_1 + \dots + \theta_k = \gamma_3} \sum_{\sigma_1 + \dots + \sigma_k = \delta_3} \frac{\gamma_3!}{\theta_1! \dots \theta_k!} \frac{\delta_3!}{\sigma_1! \dots \sigma_k!} \\
	&\times \prod_{\ell = 1}^{k} C^{\theta_\ell + \beta_\ell + \sigma_\ell + 1}_{\tilde{\Lambda}} (\theta_\ell)!^{\mu}(\alpha_\ell+\sigma_{\ell})!^{\mu} 
	\langle \xi \rangle^{-\theta_\ell} \langle x \rangle^{1-\sigma - \alpha_\ell - \sigma_{\ell}}  \\
	&\leq C_1^{\gamma + \delta + 2(\alpha + \beta) + 1}\gamma!^{\mu}\delta!^{s_0}(\alpha + \beta)!^{\mu + s_0} 
	\langle \xi \rangle^{m_1 - \gamma -(\alpha + \beta)} \langle x \rangle^{m_2 - \delta -(\alpha + \beta)} \\
	&\times \sum_{k = 1}^{\alpha} \frac{ \langle x \rangle^{k(1-\sigma)}} {k!} \sum_{h =1}^{\beta} \frac{ \langle x \rangle^{h(1-\sigma)} }{h!} \\
	&\leq C_1^{\gamma + \delta + 2(\alpha + \beta) + 1}\gamma!^{\mu}\delta!^{s_0}(\alpha + \beta)!^{\mu + s_0} 
	\langle \xi \rangle^{m_1 - \gamma -(\alpha + \beta)} \langle x \rangle^{m_2 - \delta -(\alpha + \beta)} \\
	&\times 2^{\alpha+\beta} \sum_{k = 1}^{\alpha+\beta} \frac{ \langle x \rangle^{k(1-\sigma)}} {k!}.
\end{align*}
The above estimate implies
$$
|\partial^{\alpha}_{\xi} \partial^{\beta}_{x} c_j(x,\xi)| \leq  C^{\alpha + \beta + 2j + 1}\alpha!^{\mu}\beta!^{s_0}(j)!^{\mu + s_0-1} 
\langle \xi \rangle^{m_1 - \alpha - j} \langle x \rangle^{m_2 - \beta - j} 
\sum_{k = 1}^{j} \frac{ \langle x \rangle^{k(1-\sigma)} }{k!},
$$
for every $j \geq 0$, $\alpha, \beta \in \N_0$ and $x,\xi \in \R$.

Let $\psi(x, \xi) \in C^{\infty}(\R^{2})$ such that $\psi \equiv 0$ on $Q_{2}$, $\psi \equiv 1$ on $Q^{e}_{3}$, $0 \leq \psi \leq 1$ and 
$$
|\partial^{\alpha}_{\xi} \partial^{\beta}_{x} \psi (x, \xi)| \leq C^{\alpha + \beta + 1}\alpha!^{\mu} \beta!^{s_0},
$$
for every $x, \xi \in \R$ and $\alpha, \beta \in \N_0$. Now set $\psi_0 \equiv 1$ and, for $j \geq 1$, 
$$
\psi_j (x, \xi) := \psi \left( \dfrac{x}{R(j)}, \dfrac{\xi}{R(j)} \right),
$$
where $R(j) = R j^{s_0 + \mu - 1}$, for a large constant $R > 0$. 

Setting $b(x,\xi) = \sum_{j \geq 0} \psi_{j}(x,\xi)c_{j}(x,\xi)$ we have $b \in \SG^{m_1, \infty}_{\mu, s_0}(\R^{2})$ and 
$$
b(x,\xi) \sim \sum_{j \geq 0} c_j(x,\xi) \,\, \text{in} \,\, FSG^{m_1, \infty}_{\mu, s_0}(\R^{2}).
$$
By similar arguments as the ones used in Section \ref{section_invertibility_of_e_power_tilde_Lambda} we can prove that 
$$
\sum_{j \geq k} \psi_{j}(x,\xi) c_{j}(x,\xi) \in \SG^{m_1-k, m_2 - \sigma k}_{\mu,s_0}(\R^{2}), \quad k \in \N_{0}.
$$
Hence 
$$
b(x,\xi) - \sum_{j < k} c_j(x,\xi) \in \SG^{m_1-k, m_2 - \sigma k}_{\mu, s_0}, \quad k \in \N.
$$
Since $1/(1-\sigma) > \mu + s_0 -1$ we can conclude that $b-a_3 \in \mathcal{S}_{\mu+s_0-1}(\R^{2})$ and we obtain 
$$
e^{\tilde{\Lambda}} \circ a \circ ^{R}e^{-\tilde{\Lambda}}  = b(x,D) + \tilde{r}(x,D),
$$
where $\tilde{r} \in \mathcal{S}_{\mu+s_0-1}(\R^{2})$. This concludes the proof.
\end{proof}

\subsection{Conjugation by $e^{\tilde{\Lambda}}$}

We start noting that $e^{\tilde{\Lambda}} \partial_t \{e^{\tilde{\Lambda}}\}^{-1} = \partial_t$ since $\tilde{\Lambda}$ does not depend on $t$.

\begin{itemize}
\item Conjugation of $ia_3(t,D)$. 
\\
Since $a_3$ does not depend of $x$, applying Lemma \ref{lemma_conjugation_by_the_reverse_operator}, we have
$$
e^{\tilde{\Lambda}}(x,D) (ia_3)(t,D) ^{R}(e^{-\tilde{\Lambda}})= ia_3(t,D) + s(t,x,D) + r_3(t,x,D),
$$
with
$$s \sim \sum_{j \geq 1} \frac{1}{j!} \partial^{j}_{\xi} \{e^{\tilde{\Lambda}} ia_3 D^{j}_{x}e^{-\tilde{\Lambda}}\} \in FSG^{2, -\sigma}_{\mu, s_0}(\R^{2}),\quad r_3 \in C([0,T],\mathcal{S}_{\mu+s_0-1}(\R^{2})).$$ Hence, using \eqref{equation_inverse_of_e_power_tilde_Lambda_in_a_precise_way} we can write more explicitly $s(t,x,D)$ and we obtain 
\begin{align*}
e^{\tilde{\Lambda}} (ia_3)\{e^{\tilde{\Lambda}}\}^{-1} 
&= \textrm{op}\left( ia_3 - \partial_{\xi}(a_3 \partial_{x}\tilde{\Lambda}) + \frac{i}{2}\partial^{2}_{\xi}[a_3(\partial^{2}_{x}\tilde{\Lambda} - (\partial_{x}\tilde{\Lambda})^2)] + a^{(0)}_{3} + r_3\right) \\ 
&\circ \left[ I +\textrm{op}\left(- i\partial_{\xi} \partial_{x} \tilde{\Lambda} - \frac{1}{2}\partial^{2}_{\xi}(\partial^{2}_{x}\tilde{\Lambda} - [\partial_x \tilde{\Lambda}]^{2}) - [\partial_{\xi} \partial_{x} \tilde{\Lambda} ]^{2} + q_{-3} \right)\right] \\
&= ia_3 - \textrm{op}\left( \partial_{\xi}(a_3 \partial_{x}\tilde{\Lambda}) + \frac{i}{2}\partial^{2}_{\xi}\{a_3(\partial^{2}_{x}\tilde{\Lambda} - \{\partial_{x}\tilde{\Lambda}\}^2)\}
- a_3\partial_{\xi}\partial_{x}\tilde{\Lambda} + i\partial_{\xi}a_3\partial_{\xi}\partial^{2}_{x}\tilde{\Lambda} \right) 
\\
&+ \textrm{op}\left( i\partial_{\xi}(a_3 \partial_{x}\tilde{\Lambda})\partial_{\xi}\partial_{x}\tilde{\Lambda} - \frac{i}{2}a_3\{\partial^{2}_{\xi}(\partial^{2}_{x}\tilde{\Lambda} + [\partial_x \tilde{\Lambda}]^{2}) + 2[\partial_{\xi} \partial_{x} \tilde{\Lambda} ]^{2}+ r_0 + \bar{r}\right) \\
&= ia_3 - \textrm{op}\left( \partial_{\xi}a_3 \partial_{x}\tilde{\Lambda} + \frac{i}{2}\partial^{2}_{\xi}\{a_3[\partial^{2}_{x}\tilde{\Lambda} - 
(\partial_{x}\tilde{\Lambda})^2]\} + i\partial_{\xi}a_3\partial_{\xi}\partial^{2}_{x}\tilde{\Lambda}\right)
\\
&+\textrm{op}\left( i\partial_{\xi}(a_3 \partial_{x}\tilde{\Lambda})\partial_{\xi}\partial_{x}\tilde{\Lambda} - \frac{i}{2}a_3\{\partial^{2}_{\xi}(\partial^{2}_{x}\tilde{\Lambda} + [\partial_x \tilde{\Lambda}]^{2}) + 2(\partial_{\xi} \partial_{x} \tilde{\Lambda} )^{2}\}+ r_0+ \bar{r}\right)\\
&= ia_3 - \textrm{op}\left( \partial_{\xi}a_3 \partial_{x}\lambda_2 + \partial_{\xi}a_3\partial_{x}\lambda_1 + \frac{i}{2}\partial^{2}_{\xi}\{a_3(\partial^{2}_{x}\lambda_2 - \{\partial_{x}\lambda_2\}^2)\} + i\partial_{\xi}a_3\partial_{\xi}\partial^{2}_{x}\lambda_2 \right) \\
&+\textrm{op}\left( i\partial_{\xi}(a_3 \partial_{x}\lambda_2)\partial_{\xi}\partial_{x}\lambda_2 - \frac{i}{2}a_3\{\partial^{2}_{\xi}(\partial^{2}_{x}\lambda_2 + [\partial_x \lambda_2]^{2}) + 2[\partial_{\xi} \partial_{x} \lambda_2 ]^{2}\}  + r_0+ \bar{r}\right),
\end{align*}
where $a^{(0)}_{3} \in C([0,T]; \SG^{0,0}_{\mu, s_0}(\R^2))$, $r_0 \in C([0,T]; \SG^{0,0}_{\delta}(\R^{2}))$ and, since we may assume $\delta < \mu+s_0-1$, $\bar{r} \in C([0,T];\mathcal{S}_{\mu+s_0-1}(\R^2))$.

\item Conjugation of $ia_2(t,x,D)$. 
\\
By Lemma \ref{lemma_conjugation_by_the_reverse_operator} with $M=2$ and using \eqref{equation_inverse_of_e_power_tilde_Lambda_in_a_precise_way}  we get
\begin{align*}
e^{\tilde{\Lambda}}(x,D) (ia_2)(t,x,D)\{e^{\tilde{\Lambda}}(x,D)\}^{-1} &= \textrm{op}\left(ia_2 - \partial_{\xi}\{a_2 \partial_{x}\tilde{\Lambda}\} + \partial_{\xi}\tilde{\Lambda}\partial_{x}a_2 + a^{(0)}_{2} + r_2 \right) \\ 
& \circ \left[ I -\textrm{op}\left( i\partial_{\xi} \partial_{x} \tilde{\Lambda} + \frac{1}{2}\partial^{2}_{\xi}(\partial^{2}_{x}\tilde{\Lambda} - [\partial_x \tilde{\Lambda}]^{2}) + [\partial_{\xi} \partial_{x} \tilde{\Lambda} ]^{2} + q_{-3}\right)\right] \\
&= ia_2(t,x,D) + \textrm{op}( - \partial_{\xi}\{a_2 \partial_{x}\tilde{\Lambda}\} + \partial_{\xi}\tilde{\Lambda}\partial_{x}a_2 + a_2\partial_{\xi}\partial_{x} \tilde{\Lambda} + r_0+ \bar{r})\\
&= ia_2(t,x,D) + \textrm{op}(- \partial_{\xi}a_2 \partial_{x}\tilde{\Lambda} + \partial_{\xi}\tilde{\Lambda}\partial_{x}a_2 + r_0 + \bar{r} ) \\
&= ia_2(t,x,D) + \textrm{op}\left( - \partial_{\xi}a_2 \partial_{x}\lambda_2 + \partial_{\xi}\lambda_2\partial_{x}a_2 + r_0 +\bar{r} \right),
\end{align*}
where $a^{(0)}_{2} \in C([0,T]; \SG^{0,0}_{\mu, s_0})$, $r_0 \in C([0,T]; \SG^{0,0}_{\delta}(\R^{2}))$ and $\bar{r} \in C([0,T];\Sigma_{\mu+s_0-1} (\R^{2}))$.

\item Conjugation of $ia_1(t,x,D)$:
\begin{align*}
e^{\tilde{\Lambda}}(x,D) ia_1(t,x,D) \{e^{\tilde{\Lambda}}(x,D)\}^{-1} = \textrm{op}(ia_1 + a^{(0)}_{1} + r_1)\circ \sum_{j\geq 0}(-r(t, x,D))^{j} =\textrm{op}( ia_1+ \tilde{r}_{0} + \tilde{r}), 
\end{align*}
\noindent 
where $a^{(0)}_{1} \in C([0,T]; \SG^{0,1-2\sigma}_{\mu, s_0}(\R^2))$, $\tilde{r_0} \in C([0,T]; \SG^{0,1-2\sigma}_{\delta}(\R^{2}))$, $\tilde{r} \in C([0,T];\Sigma_{\mu+s_0-1} (\R^{2}))$.

\item Conjugation of $ia_0(t,x,D)$:
\begin{align*}
e^{\tilde{\Lambda}}(x,D) ia_0(t,x,D)\{e^{\tilde{\Lambda}}(x,D)\}^{-1} = \textrm{op}(ia_0 + a^{(0)}_{0} + r_0) \sum_{j\geq 0}(-r(t,x,D))^{j} = \textrm{op}( ia_0+ \tilde{\tilde{r}}_{0} + \tilde{r}_1), 
\end{align*}
\noindent
where $a^{(0)}_{0} \in C([0,T]; \SG^{0,1-2\sigma}_{\mu, s_0})$, $\tilde{\tilde{r}}_0 \in C([0,T]; \SG^{-1,1-2\sigma}_{\delta}(\R^{2}))$ and $\tilde{r}_1 \in C([0,T];\Sigma_{\mu+s_0-1} (\R^{2}))$. 
\end{itemize}

Summing up we obtain
\begin{multline}\label{conj_tilde_Lambda}
e^{\tilde{\Lambda}}(x,D) (iP(t,x,D_t,D_x)) \{e^{\tilde{\Lambda}}(x,D)\}^{-1} = \partial_{t} + ia_3(t,D) + ia_2(t,x,D) \\- \text{op}(\partial_{\xi}a_3\partial_{x}\lambda_2) 
+ ia_1(t,x,D) -\textrm{op}( \partial_{\xi}a_3\partial_{x}\lambda_1 + \partial_{\xi}a_2\partial_{x}\lambda_2 - \partial_{\xi}\lambda_2\partial_{x}a_2 +ib_1)\\ + ia_0(t,x,D) + r_{\sigma}(t,x,D) + r_{0}(t,x,D) + \bar{r}(t,x,D), 
\end{multline}
where 
\beqs\label{b1}
b_1 \in C([0,T]; \SG^{1,-2\sigma}_{\mu, s_0}(\R^{2})),\ b_1(t,x,\xi)\in\R,\ b_1\ {\rm does\ not\ depend\ on}\ \lambda_1,
\eeqs
and 
$$
r_{0} \in C([0,T]; \SG^{0,0}_{\delta}(\R^{2})), \; r_{\sigma} \in C([0,T]; \SG^{0, 1-2\sigma}_{\delta}(\R^{2})),\; \bar{r} \in C([0,T]; \Sigma_{\mu+s_0-1}(\R^{2})).
$$

\subsection{Conjugation by $e^{k(t)\langle x \rangle^{1-\sigma}_{h}}$}

Let us recall that we are assuming that $k \in C^{1}([0,T]; \R)$, $k'(t)\leq 0$ and $k(t) \geq 0$ for every $t\in [0,T]$. Following the same ideas of the proof of Lemma \ref{lemma_conjugation_by_the_reverse_operator} one can prove the following result.

\begin{lemma}\label{lem7}
Let $a \in C([0,T],\SG^{m_1, m_2}_{\mu, s_0}(\R^2))$, where $1 < \mu < s_0$ and $\mu + s_0 - 1 < \frac{1}{1-\sigma}$. Then
$$
e^{k(t)\langle x \rangle^{1-\sigma}_{h}} \, a(t,x,D) \, e^{-k(t)\langle x \rangle^{1-\sigma}_{h} } = a(t,x,D) + b(t,x,D) + \bar{r}(t,x,D),
$$ 
where $b \sim \sum_{j \geq 1} \frac{1}{j!} e^{k(t)\langle x \rangle^{1-\sigma}_{h}}\, \partial^{j}_{\xi} a\, D^{j}_{x} e^{-k(t)\langle x \rangle^{1-\sigma}_{h}}$ in  $\SG^{m_1 - 1, m_2 -\sigma}_{\mu, s_0}(\R^{2})$ and $\bar{r} \in C([0,T],\mathcal{S}_{\mu+s_0-1}(\R^{2}))$.
\end{lemma} 
Let us perform the conjugation by $e^{k(t)\langle x \rangle^{1-\sigma}_{h}}$ of the operator $e^{\tilde{\Lambda}} (iP) \{e^{\tilde{\Lambda}}\}^{-1}$ in \eqref{conj_tilde_Lambda} with the aid of Lemma \ref{lem7}.
\begin{itemize}
\item Conjugation of $ \partial_t$: $e^{k(t)\langle x \rangle^{1-\sigma}_{h}  } \, \partial_{t} \, e^{-k(t)\langle x \rangle^{1-\sigma}_{h}} = \partial_{t} - k'(t)\langle x \rangle^{1-\sigma}_{h}$.

\item Conjugation of $ia_3(t,D)$: 
\begin{align*}
e^{k(t)\langle x \rangle^{1-\sigma}_{h}} \, ia_3(t,D) e^{-k(t)\langle x \rangle^{1-\sigma}_{h}} &= 
ia_3 +\textrm{op}(- k(t)\partial_{\xi}a_3 \partial_{x} \langle x \rangle^{1-\sigma}_{h}) \\ 
&+ \textrm{op}\left(\frac{i}{2}\partial^{2}_{\xi}a_3 \{k(t)\partial^2_{x}\langle x \rangle^{1-\sigma}_{h}
- k^{2}(t) [\partial_{x}\langle x \rangle^{1-\sigma}_{h}]^2 \}\right)  + a^{(0)}_{3} + r_{3},
\end{align*}
where $a^{(0)}_{3} \in C([0,T]; \SG^{0, -3\sigma}_{\mu, s_0}(\R^{2}))$ and $r_{3} \in C([0,T]; \Sigma_{\mu+s_0-1}(\R^{2}))$.

\item Conjugation of $\textrm{op}(ia_2 -\partial_{\xi}a_3 \partial_{x}\lambda_2)$:
\begin{align*}
&e^{k(t)\langle x \rangle^{1-\sigma}_{h}} \textrm{op}(ia_2 -\partial_{\xi}a_3\partial_{x}\lambda_2) \,e^{-k(t)\langle x \rangle^{1-\sigma}_{h}} 
\\
&=
ia_2 +\textrm{op}(-\partial_{\xi}a_3\partial_{x}\lambda_2 )
+\textrm{op}(- k(t)\partial_{\xi}a_2 \partial_{x} \langle x \rangle^{1-\sigma}_{h} - ik(t)\partial_{\xi}\{\partial_{\xi}a_3\partial_{x}\lambda_2\} \partial_{x} \langle x \rangle^{1-\sigma}_{h} + a^{(0)}_{2}+ r_2),
\end{align*}
where $a^{(0)}_{2} \in C([0,T]; \SG^{0, -2\sigma}_{\mu, s_0}(\R^{2}))$ and $r_2 \in C([0,T]; \Sigma_{\mu+s_0-1}(\R^{2}))$.

\item Conjugation of $i(a_1+a_0)(t,x,D)$: 
$$e^{k(t)\langle x \rangle^{1-\sigma}_{h}} i(a_1+a_0)(t,x,D)e^{-k(t)\langle x \rangle^{1-\sigma}_{h}} = \textrm{op}(ia_1 + ia_0 + a_{1,0}+ r_1),$$
where  $r_1 \in C([0,T]; \Sigma_{\mu+s_0-1}(\R^{2}))$ and
$$
a_{1, 0} \sim \sum_{j \geq 1} e^{k(t)\langle x \rangle^{1-\sigma}_{h}} \frac{1}{j!} \partial^{j}_{\xi}i(a_1+a_0) D^{j}_{x} e^{-k(t)\langle x \rangle^{1-\sigma}_{h}} \,\, \text{in} \,\, \SG^{0, 1-2\sigma}_{\mu, s_0}(\R^{2}).
$$ 
It is not difficult to verify that the following estimate holds
\beqs\label{a10}
| a_{1,0} (t,x,\xi)| \leq \max\{1, k(t)\} C_{T}  \langle x \rangle^{1 - 2\sigma}_{h},
\eeqs
where $C_{T}$ depends on $a_1$ and does not depend on $k(t)$.

\item Conjugation of $\textrm{op}(- \partial_{\xi}a_3\partial_{x}\lambda_1 - \partial_{\xi}a_2\partial_{x}\lambda_2 + \partial_{\xi}\lambda_2\partial_{x}a_2 + ib_1)$: taking into account i) of Lemma \ref{lemma_estimates_lambda_1}
\begin{align*}
e^{k(t)\langle x \rangle^{1-\sigma}_{h}} &\textrm{op}(- \partial_{\xi}a_3\partial_{x}\lambda_1 - \partial_{\xi}a_2\partial_{x}\lambda_2 + \partial_{\xi}\lambda_2\partial_{x}a_2 + ib_1) e^{-k(t)\langle x \rangle^{1-\sigma}_{h}} \\
&= \textrm{op}( -\partial_{\xi}a_3\partial_{x}\lambda_1 
- \partial_{\xi}a_2\partial_{x}\lambda_2 + \partial_{\xi}\lambda_2\partial_{x}a_2
+ ib_1 + r_0 + \bar{r}),
\end{align*}
where $r_0 \in C([0,T]; \SG^{0, 0}_{\mu, s_0}(\R^{2}))$ and $\bar{r} \in C([0,T]; \Sigma_{\mu+s_0-1}(\R^{2}))$.

\item Conjugation of $r_{\sigma}(t,x,D)$:  
$$e^{k(t)\langle x \rangle^{1-\sigma}_{h}} \,r_{\sigma}(t,x,D) \, e^{-k(t)\langle x \rangle^{1-\sigma}_{h}} = r_{\sigma,1}(t,x,D) + \bar{r}(t,x,D),$$ 
where $r \in C([0,T]; \Sigma_{\mu+s_0-1}(\R^{2}))$, $r_{\sigma,1}\in C([0,T]; \SG^{0, 1-2\sigma}_{\delta}(\R^{2}))$ and the estimate 
\beqs\label{erresigma1}
|r_{\sigma,1}(t,x,\xi)| \leq C_{T,\tilde{\Lambda}} \langle x \rangle^{1-2\sigma}_{h}
\eeqs
holds with $C_{T,\tilde{\Lambda}}$ independent of $k(t)$.
\end{itemize}

Gathering all the previous computations we may write
\beqs\label{conj_k}
e^{k(t)\langle x \rangle^{1-\sigma}_{h}} &\!\!\!e^{\Lambda}&\!\!\! (iP) \{e^{\Lambda}\}^{-1}e^{-k(t)\langle x \rangle^{1-\sigma}_{h}} = \partial_{t} + \textrm{op}(ia_3 -\partial_{\xi}a_3\partial_{x}\lambda_2  + ia_2- k(t)\partial_{\xi}a_3 \partial_{x} \langle x \rangle^{1-\sigma}_{h}) \\\nonumber
&+& \textrm{op}(-\partial_{\xi}a_3\partial_{x}\lambda_1  + ia_1 - \partial_{\xi}a_2\partial_{x}\lambda_2 + \partial_{\xi}\lambda_2\partial_{x}a_2 
- k(t)\partial_{\xi}a_2 \partial_{x} \langle x \rangle^{1-\sigma}_{h} ) \\\nonumber
&+&\textrm{op}( ib_1 + ic_1 +ia_0 - k'(t)\langle x \rangle^{1-\sigma}_{h} + a_{1,0}  + r_{\sigma,1}) + r_0(t,x,D) + \bar{r}(t,x,D), 
\eeqs
where $b_1$ satisfies \eqref{b1}, 
\beqs\label{c1}
c_1 \in C([0,T]; \SG^{1, -2\sigma}_{\mu, s_0}(\R^{2})),\ c_1(t,x,\xi)\in\R,\ c_1\ {\rm does\ not\ depend\ on}\ \lambda_1,
\eeqs
(but $c_1$ depends of $\lambda_2, k(t)$), $a_{1,0}$ as in \eqref{a10}, $r_{\sigma,1}$ as in \eqref{erresigma1}, and for some new operators
\beqsn
r_0 \in C([0,T]; \SG^{0,0}_{\delta}(\R^{2})),\; \bar{r} \in C([0,T]; \mathcal{S}_{\mu + s_0 -1}(\R^{2})).
\eeqsn


\subsection{Conjugation by $e^{\rho_1 \langle D \rangle^{\frac{1}{\theta}}}$}

Since we are considering $\theta > s_0$ and $\mu > 1$ arbitrarily close to $1$, we may assume that all the previous smoothing remainder terms have symbols  in $\Sigma_{\theta}(\R^{2})$.
\begin{lemma}
	Let $a \in \SG^{m_1, m_2}_{\mu, s_0}$, where $1 < \mu < s_0$ and $\mu + s_0 - 1 < \theta$. Then
	$$
	e^{\rho_1 \langle D \rangle^{\frac{1}{\theta}}} \, a(x,D) \, e^{-\rho_1 \langle D \rangle^{\frac{1}{\theta}} } = a(x,D) + b(x,D) + r(x,D),
	$$ 
	where $b \sim \sum_{j \geq 1} \frac{1}{j!} \partial^{j}_{\xi}e^{\rho_1 \langle \xi \rangle^{\frac{1}{\theta}}}\, D^{j}_{x}a \, e^{-\rho_1 \langle \xi \rangle^{\frac{1}{\theta}}}$ in  $\SG^{m_1 - (1-\frac{1}{\theta}), m_2 -1}_{\mu, s_0}(\R^{2})$ and $r \in \mathcal{S}_{\mu+s_0 -1}(\R^{2})$.
\end{lemma}

Let us now conjugate by $e^{\rho_1 \langle D \rangle^{\frac{1}{\theta}}}$ the operator $e^{k(t)\langle x \rangle^{1-\sigma}_{h}} e^{\Lambda} (iP) \{e^{\Lambda}\}^{-1}e^{-k(t)\langle x \rangle^{1-\sigma}_{h}}$ in \eqref{conj_k}.
First of all we observe that since $a_3$ does not depend of $x$, we simply have 
$$e^{\rho_1 \langle D \rangle^{\frac{1}{\theta}}} \, ia_3(t,D) \, e^{-\rho_1 \langle D \rangle^{\frac{1}{\theta}}} = ia_3(t,D). $$
\begin{itemize}
	\item Conjugation of $\textrm{op}\left(-\partial_{\xi}a_3\partial_{x}\lambda_2 + ia_2 - k(t)\partial_{\xi}a_3\partial_x \langle x \rangle^{1-\sigma}_{h}\right)$: 
	\begin{align*}
		e^{\rho_1 \langle D \rangle^{\frac{1}{\theta}}} \,\textrm{op}&( -\partial_{\xi}a_3\partial_{x}\lambda_2 + ia_2 - k(t)\partial_{\xi}a_3\partial_x \langle x \rangle^{1-\sigma}_{h} )\, e^{-\rho_1 \langle D \rangle^{\frac{1}{\theta}}} \\
		&= \textrm{op}(-\partial_{\xi}a_3\partial_{x}\lambda_2 + ia_2 - k(t)\partial_{\xi}a_3\partial_x \langle x \rangle^{1-\sigma}_{h}) + (a_{2,\rho_1} + \bar{r})(t,x,D),
	\end{align*} 
	where $a_{2,\rho_1} \in C([0,T], \SG^{1+\frac{1}{\theta}, -1}_{\mu, s_0}(\R^{2}))$, $\bar{r} \in C([0,T],\Sigma_{\theta}(\R^{2}))$ and the following estimate holds
	\beqs\label{a2rho1}
	|\partial^{\alpha}_{\xi} \partial^{\beta}_{x}a_{2,\rho_1}(t,x,\xi)| \leq \max\{1,k(t)\} |\rho_1| C_{\lambda_2, T}^{\alpha+\beta+1}\alpha!^{\mu}\beta!^{s_0} \langle \xi \rangle^{1+\frac{1}{\theta} -\alpha} \langle x \rangle^{-\sigma - \beta}.
	\eeqs
	In particular
	$$
	|a_{2,\rho_1}(t,x,\xi)| \leq \max\{1,k(t)\} |\rho_1| C_{\lambda_2, T} \langle \xi \rangle^{1+\frac{1}{\theta}}_{h} \langle x \rangle^{-\sigma}.
	$$
	
	\item Conjugation of
	$
	\textrm{op}(- \partial_{\xi}a_3\partial_{x}\lambda_1 + ia_1 - \partial_{\xi}a_2\partial_{x}\lambda_2 + \partial_{\xi}\lambda_2\partial_{x}a_2 - k(t)\partial_{\xi}a_2 \partial_{x} \langle x \rangle^{1-\sigma}_{h} + ib_1 + ic_1):
	$ 
	the conjugation of this term is given by  
	\begin{align*}
		\textrm{op}(- \partial_{\xi}a_3\partial_{x}\lambda_1  + &ia_1 - \partial_{\xi}a_2\partial_{x}\lambda_2 + \partial_{\xi}\lambda_2\partial_{x}a_2 
		- k(t)\partial_{\xi}a_2 \partial_{x} \langle x \rangle^{1-\sigma}_{h} + ib_1 + ic_1) \\ 
		&+ a_{1, \rho_1}(t,x,D) + \bar{r}(t,x,D),
	\end{align*}
	where $\bar{r} \in C([0,T],\Sigma_{\theta} (\R^{2}))$ and $a_{1,\rho_1}$ satisfies the following estimate
	\beqs\label{a1rho1}
	|\partial^{\alpha}_{\xi} \partial^{\beta}_{x}a_{1,\rho_1}(t,x,\xi)| \leq \max\{1,k(t)\} |\rho_1| C_{\tilde{\Lambda}, T}^{\alpha+\beta+1}\alpha!^{\mu}\beta!^{s_0} \langle \xi \rangle^{\frac{1}{\theta} -\alpha} \langle x \rangle^{-\sigma - \beta}.
	\eeqs
	In particular
	$$
	|a_{1,\rho_1}(t,x,\xi)| \leq \max\{1,k(t)\} |\rho_1| C_{\tilde{\Lambda}, T}  \langle \xi \rangle^{\frac{1}{\theta}}_{h} \langle x \rangle^{-\sigma}.
	$$
	
	\item Conjugation of $\textrm{op}(ia_0 - k'(t)\langle x \rangle^{1-\sigma}_{h} + a_{1,0}  + r_{\sigma,1})$:
	\begin{align*}
		e^{\rho_1 \langle D \rangle^{\frac{1}{\theta}}} \textrm{op}(ia_0 - &k'(t)\langle x \rangle^{1-\sigma}_{h} + a_{1,0}  + r_{\sigma,1}) e^{-\rho_1 \langle D \rangle{\frac{1}{\theta}}} \\
		& = \textrm{op}( ia_0 - k'(t)\langle x \rangle^{1-\sigma}_{h} + a_{1,0}  + r_{\sigma,1}) + r_0(t,x,D) + \bar{r}(t,x,D), 
	\end{align*}
	where $r_0 \in C([0,T]; \SG^{(0,0)}_{\delta}(\R^{2}))$ and $\bar{r} \in C([0,T]; \Sigma_{\theta}(\R^{2}))$.
\end{itemize}

Summing up, from the previous computations we obtain 
\beqs\nonumber
e^{\rho_1 \langle D \rangle^{\frac{1}{\theta}}} &&\hskip-0.5cm e^{k(t)\langle x \rangle^{1-\sigma}_{h}} e^{\tilde{\Lambda}}(iP) \{ e^{\rho_1 \langle D \rangle^{\frac{1}{\theta}}} e^{k(t)\langle x \rangle^{1-\sigma}_{h}} e^{\tilde{\Lambda}} \}^{-1}
= \partial_{t} + ia_3(t,D)
\\\nonumber
&+& \textrm{op}\left(-\partial_{\xi}a_3\partial_{x}\lambda_2  + ia_2 - k(t)\partial_{\xi}a_3 \partial_{x} \langle x \rangle^{1-\sigma}_{h} +a_{2,\rho_1}
\right) \\ \nonumber
&+& \textrm{op}\left(- \partial_{\xi}a_3\partial_{x}\lambda_1+ ia_1 - \partial_{\xi}a_2\partial_{x}\lambda_2 + \partial_{\xi}\lambda_2\partial_{x}a_2 
- k(t)\partial_{\xi}a_2 \partial_{x} \langle x \rangle^{1-\sigma}_{h} + ib_1 + ic_1 + a_{1,\rho_1}\right) \\
&+&\textrm{op}( ia_0- k'(t)\langle x \rangle^{1-\sigma}_{h} + a_{1,0}  + r_{\sigma,1}) + r_0(t,x,D) + \bar{r}(t,x,D), \nonumber
\eeqs
with $a_{2,\rho_1}$ as in \eqref{a2rho1}, $b_1$ as in \eqref{b1}, $c_1$ as in \eqref{c1}, $a_{1,\rho_1}$ as in \eqref{a1rho1}, $a_{1,0}$ as in \eqref{a10}, $r_{\sigma,1}$ as in \eqref{erresigma1}, and for some operators
$$
r_0 \in C([0,T]; \SG^{(0,0)}_{\delta}(\R^{2})), \quad \bar{r} \in C([0,T]; \Sigma_{\theta}(\R^{2})).
$$


\section{Estimates from below}\label{section_choices_of_M1_M2_kt}

In this section we will choose $M_2, M_1$ and $k(t)$ in order to apply Fefferman-Phong and sharp G\aa rding inequalities to get the desired energy estimate for the conjugated problem. By the computations of the previous section we have 
\beqsn 
e^{\rho_1 \langle D \rangle^{\frac{1}{\theta}}} e^{k(t)\langle x \rangle^{1-\sigma}_{h}} e^{\tilde{\Lambda}}(x,D) \hskip-0.5cm && (iP) \{ e^{\rho_1 \langle D \rangle^{\frac{1}{\theta}} } e^{k(t)\langle x \rangle^{1-\sigma}_{h}} e^{\tilde{\Lambda}}(x,D) \}^{-1} 
\\ \nonumber
&=& \partial_{t} + ia_3(t,D) + \sum_{j=0}^2 \tilde{a}_j(t,x,D)+ r_0(t,x,D) + \bar{r}(t,x,D),
\eeqsn
where $\tilde{a}_2, \tilde{a}_1$ represent the part with $\xi-$order $2,1$ respectively and $\tilde{a}_0$ represents the part with $\xi-$order $0$, but with a positive order (less than or equal to $1-\sigma$) with respect to $x$. The real parts are given by 
\beqsn
Re\, \tilde{a}_2 &=& -\partial_{\xi}a_3\partial_{x}\lambda_2  - Im\, a_2 
- k(t)\partial_{\xi}a_3 \partial_{x} \langle x \rangle^{1-\sigma}_{h} + Re\, a_{2,\rho_1},
\\ 
Im\, \tilde{a}_2 &=& Re\, a_2 + Im\, a_{2,\rho_1},
\\ 
	Re\, \tilde{a}_1 &=& - \partial_{\xi}a_3\partial_{x}\lambda_1 - Im\,a_1 - \partial_{\xi}Re\,a_2\partial_{x}\lambda_2 + \partial_{\xi}\lambda_2\partial_{x}Re\,a_2 
	\\ 
	&&- k(t)\partial_{\xi}Re\,a_2 \partial_{x} \langle x \rangle^{1-\sigma}_{h} 
	+ Re\,a_{1,\rho_1},
\\ 
Re\, \tilde{a}_0 &=& -Im a_0- k'(t)\langle x \rangle^{1-\sigma}_{h} + Re\,{a}_{1,0}  + Re\,{r}_{\sigma,1}.  
\eeqsn
Since the Fefferman-Phong inequality holds true only for scalar symbols, we need to decompose $Im\, \tilde{a}_2$ into its Hermitian and anti-Hermitian part: 
$$
i Im \tilde{a}_2=\ds\frac{i Im \tilde{a}_2+(i Im \tilde{a}_2)^*}{2}+\frac{i Im \tilde{a}_2-(i Im \tilde{a}_2)^*}{2}=t_1+t_2,
$$
where $2Re\langle t_2(t,x,D)u,u\rangle =0$ and $t_1(t,x,\xi)=-\ds\sum_{\alpha\geq 1}\frac{i}{2\alpha!}\partial_\xi^\alpha D_x^\alpha Im \tilde{a}_2(t,x,\xi)$. Observe that, using  \eqref{a2rho1}, 
\begin{align}\label{aces_high} \nonumber
|t_1(t,x,\xi)| &\leq C_{a_2} \langle \xi \rangle \langle x \rangle^{-1} + \max\{1, k(t)\}|\rho_1|C_{\lambda_2} \langle \xi \rangle^{\frac{1}{\theta}} \langle x \rangle^{-\sigma - 1}  \\
&\leq \{ C_{a_2} + h^{-(1-\frac{1}{\theta})}\max\{1, k(0)\}|\rho_1|C_{\lambda_2} \} \langle \xi \rangle_{h} \langle x \rangle^{-\frac{\sigma}{2}}.
\end{align}
In this way we may write
\beqs\label{finalP}
e^{\rho_1 \langle D \rangle^{\frac{1}{\theta}}}  \hskip-0.5cm && e^{k(t)\langle x \rangle^{1-\sigma}_{h}} e^{\tilde{\Lambda}}(x,D) (iP) \{ e^{\rho_1 \langle D \rangle^{\frac{1}{\theta}} } e^{k(t)\langle x \rangle^{1-\sigma}_{h}} e^{\tilde{\Lambda}}(x,D) \}^{-1} 
\\\nonumber
&=& \partial_{t} + ia_3(t,D) + (Re\, \tilde{a}_2 +  t_2 + t_1 + \tilde{a}_{1} + \tilde{a}_0)(t,x,D) + \tilde{r}_0(t,x,D),
\eeqs
where $\tilde{r}_0$ has symbol of order $(0,0)$.

 Now we are ready to choose $M_2, M_1$ and $k(t)$. The function $k(t)$ will be of the form $k(t)=K(T-t)$, $t \in [0,T]$, for a positive constant $K$ to be chosen. In  the following computations we shall consider $|\xi| > hR_{a_3}$. Observe that $2|\xi|^{2} \geq \langle \xi \rangle^{2}_{h}$ whenever $|\xi| \geq h > 1$. For $Re\, \tilde{a}_2$ we have:
\begin{align*}
	Re\, \tilde{a}_2 &= M_2 |\partial_{\xi} a_3| \langle x \rangle^{-\sigma} - Im\, a_2 
	- k(t)\partial_{\xi}a_3 \partial_{x} \langle x \rangle^{1-\sigma}_{h} + Re\, a_{2,\rho_1} 
	\\
	&\geq M_2 C_{a_3} |\xi|^{2} \langle x \rangle^{-\sigma} - C_{a_2} \langle \xi \rangle^{2}_h \langle x \rangle^{-\sigma}
	\\ 
	&- \tilde{C}_{a_3}k(0)(1-\sigma)\langle \xi \rangle^{2}_{h} \langle x \rangle^{-\sigma}_{h} 
	- \max\{1,k(0)\} C_{\lambda_2, \rho_1} \langle \xi \rangle^{1+\frac{1}{\theta}}_{h} \langle x \rangle^{-\sigma} 
	\\
	&\geq (M_2 \frac{C_{a_3}}{2} - C_{a_2} 
	-\tilde{C}_{a_3} k(0)(1-\sigma)  - \max\{1,k(0)\}C_{\lambda_2, \rho_1} \langle \xi \rangle^{-(1-\frac{1}{\theta})}_{h} ) \langle \xi \rangle^{2}_{h} \langle x \rangle^{-\sigma} 
	\\
	&\geq (M_2 \frac{C_{a_3}}{2} - C_{a_2} 
	- \tilde{C}_{a_3}k(0)(1-\sigma)  - \max\{1,k(0)\} C_{\lambda_2, \rho_1} h^{-(1-\frac{1}{\theta})} ) \langle \xi \rangle^{2}_{h} \langle x \rangle^{-\sigma}
	\\
	&=(M_2 \frac{C_{a_3}}{2} - C_{a_2} - \tilde{C}_{a_3}KT(1-\sigma)  - \max\{1,KT\} C_{\lambda_2, \rho_1} h^{-(1-\frac{1}{\theta})} ) 
	\langle \xi \rangle^{2}_{h} \langle x \rangle^{-\sigma}
\end{align*}

For $Re\,\tilde{a}_1$, we have:
\begin{align*}
	Re\, \tilde{a}_1 &= M_1 |\partial_{\xi}a_3| \langle \xi \rangle^{-1}_{h} \langle x \rangle^{-\frac{\sigma}{2}} 
	\psi\left( \frac{\langle x \rangle^{\sigma}}{\langle \xi \rangle^{2}_{h}} \right) 
	- Im\,a_1 - \partial_{\xi}Re\,a_2\partial_{x}\lambda_2 + \partial_{\xi}\lambda_2\partial_{x}Re\,a_2 \\
	&
	- k(t)\partial_{\xi}Re\,a_2 \partial_{x} \langle x \rangle^{1-\sigma}_{h} 
	+ Re\,a_{1,\rho_1} 
	\\
	&\geq M_1 C_{a_3} |\xi|^{2}  \langle \xi \rangle^{-1}_{h} \langle x \rangle^{-\frac{\sigma}{2}} 
	\psi\left( \frac{\langle x \rangle^{\sigma}}{\langle \xi \rangle^{2}_{h}} \right) - C_{a_1}\langle \xi \rangle_{h} \langle x \rangle^{-\frac{\sigma}{2}} - \tilde{C}_{a_2, \lambda_2} \langle \xi \rangle_{h} \langle x \rangle^{-\sigma}
	\\
	&-Ck(0) (1-\sigma) \langle \xi \rangle_{h} \langle x \rangle^{-\sigma}_{h} - \max\{1,k(0)\}C_{\tilde{\Lambda}, \rho_1} \langle \xi \rangle^{\frac{1}{\theta}}_{h} 
	\langle x \rangle^{-\sigma}_{h} \\
	&\geq M_1 \frac{C_{a_3}}{2} \langle \xi \rangle_{h} \langle x \rangle^{-\frac{\sigma}{2}} 
	\psi\left( \frac{\langle x \rangle^{\sigma}}{\langle \xi \rangle^{2}_{h}} \right) - C_{a_1}\langle \xi \rangle_{h} \langle x \rangle^{-\frac{\sigma}{2}} - \tilde{C}_{a_2, \lambda_2}\langle \xi \rangle_{h} \langle x \rangle^{-\frac{\sigma}{2}}
	\\
	&-Ck(0) (1-\sigma) \langle \xi \rangle_{h} \langle x \rangle^{-\frac{\sigma}{2}}_{h} \langle x \rangle^{-\frac{\sigma}{2}} - \max\{1,k(0)\}C_{\tilde{\Lambda}, \rho_1} \langle \xi \rangle^{- (1-\frac{1}{\theta})}_{h} \langle \xi \rangle_{h} \langle x \rangle^{-\frac{\sigma}{2}} \\
	&= (M_1 \frac{C_{a_3}}{2} - C_{a_1} - \tilde{C}_{a_2, \lambda_2} 
	- C KT(1-\sigma)\langle x \rangle^{-\frac{\sigma}{2} }_{h}) \langle \xi \rangle_{h} \langle x \rangle^{-\frac{\sigma}{2}} \\  
	&-
	\max\{1,KT\}C_{\tilde{\Lambda}, \rho_1} \langle \xi \rangle^{-(1-\frac{1}{\theta})}_{h} \langle \xi \rangle_{h} \langle x \rangle^{-\frac{\sigma}{2}} - M_1 \frac{C_{a_3}}{2} \langle \xi \rangle_{h} \langle x \rangle^{-\frac{\sigma}{2}}
	\left( 1 - \psi\left( \frac{\langle x \rangle^{\sigma}}{\langle \xi \rangle^{2}_{h}} \right)  \right).
\end{align*}
Since $\langle \xi \rangle_{h} \langle x \rangle^{-\frac{\sigma}{2}} \leq \sqrt{2}$ on the support of 
$1 - \psi\left( \frac{\langle x \rangle^{\sigma}}{\langle \xi \rangle^{2}_{h}} \right)$, we may conclude
\begin{align*}
	Re\, \tilde{a}_1 &\geq (M_1 \frac{C_{a_3}}{2} - C_{a_1} - \tilde{C}_{a_2, \lambda_2}
	- CKT(1-\sigma)h^{-\frac{\sigma}{2}})\langle \xi \rangle_{h} \langle x \rangle^{-\frac{\sigma}{2}} \\ 
	&-
	\max\{1,KT\}C_{\tilde{\Lambda}, \rho_1} h^{- (1-\frac{1}{\theta})} \langle \xi \rangle_{h} \langle x \rangle^{-\frac{\sigma}{2}} 
	- M_1 \frac{C_{a_3}}{2} \sqrt{2}.
\end{align*}
Taking \eqref{aces_high} into account we obtain
\begin{align*}
	Re\, (\tilde{a}_1 + t_1) &\geq (M_1 \frac{C_{a_3}}{2} - C_{a_2} - C_{a_1} - \tilde{C}_{a_2, \lambda_2}
	- CKT(1-\sigma)h^{-\frac{\sigma}{2}})\langle \xi \rangle_{h} \langle x \rangle^{-\frac{\sigma}{2}} \\ 
	&-
	\max\{1,KT\}(C_{\tilde{\Lambda}, \rho_1} +|\rho_1|C_{\lambda_2} )h^{- (1-\frac{1}{\theta})} \langle \xi \rangle_{h} \langle x \rangle^{-\frac{\sigma}{2}} 
	- M_1 \frac{C_{a_3}}{2} \sqrt{2}.
\end{align*}

For $Re\, \tilde{a}_0$, we have:
\begin{align*}
	Re\, \tilde{a}_{0} &= - Im a_0- k'(t)\langle x \rangle^{1-\sigma}_{h} + Re\,{a}_{1,0}  + Re\,{r}_{\sigma,1} 
	\\
	&=  - Im a_0+ ( -k'(t) - \max\{1,k(0)\}C_T \langle x \rangle^{-\sigma}_{h} - C_{T, \tilde{\Lambda}} \langle x \rangle^{-\sigma}_{h}
	)\langle x \rangle^{1-\sigma}_{h} \\
	&\geq  ( -C_{a_0}+K - \max\{1,KT\}C_Th^{-\sigma} - C_{T, \tilde{\Lambda}}h^{-\sigma} 
	)\langle x \rangle^{1-\sigma}_{h}. 
\end{align*}

Finally, let us prooced with the choices of $M_1, M_2$ and $K$. First we choose $K$ larger than $\max\{C_{a_0}, 1/T\}$, then we set
$M_2$ large in order to obtain 
$M_2 \frac{C_{a_3}}{2} - C_{a_2} - \tilde{C}_{a_3}KT(1-\sigma) > 0$ 
and after that we take $M_1$ such that
$M_1 \frac{C_{a_3}}{2} - C_{a_2} - C_{a_1} - \tilde{C}_{a_2, \lambda_2} > 0$ (choosing $M_2$, $M_1$ we determine $\tilde{\Lambda}$). 
Enlarging the parameter $h$ we may assume
$$
	KTC_{\lambda_2, \rho_1} h^{-(1-\frac{1}{\theta})} < \frac{1}{4}(M_2 \frac{C_{a_3}}{2} - C_{a_2} - \tilde{C}_{a_3}KT(1-\sigma)),
$$
$$
	CKT(1-\sigma)h^{-\frac{\sigma}{2}} + KT(C_{\tilde{\Lambda}, \rho_1} + |\rho_1|C_{\lambda_2} )h^{- (1-\frac{1}{\theta})} < \frac{1}{4}(M_1 \frac{C_{a_3}}{2} - C_{a_2} - C_{a_1} - \tilde{C}_{a_2, \lambda_2} 
	),
$$
$$
	KTC_Th^{-\sigma} + C_{T, \tilde{\Lambda}}h^{-\sigma} <\frac{K-C_{a_0}}{4}.
$$
With these choices we obtain that $Re\, \tilde{a}_2 \geq 0$, $Re\, (\tilde{a}_1 + t_1) + M_1 \frac{C_{a_3}}{2} \sqrt{2} \geq 0$ and $Re\, \tilde{a}_0 \geq 0$.
Let us also remark that the choices of $M_2, M_1$ and $k(t)$ do not depend of $\rho_1$ and $\theta$.

\section{Proof of Theorem \ref{theorem_main_theorem}}\label{section_proof_of_main_theorem}

Let us denote
\beqsn
\tilde {P}_{\Lambda}&:=&e^{\rho_1 \langle D \rangle^{\frac{1}{\theta}} } e^{\Lambda }(t,x,D) \,iP\, \{ e^{\rho_1 \langle D \rangle^{\frac{1}{\theta}} } e^{\Lambda}(t,x,D) \}^{-1}.
\eeqsn
By \eqref{finalP}, with the  choices of $M_2,M_1,k(t)$ in the previous section, we get
\beqsn
i\tilde P_{\Lambda}=\partial_{t} + ia_3 (t,D) + (Re\, \tilde{a}_2  + t_2)(t,x,D) +( \tilde{a}_{1} + t_1)(t,x,D)+ \tilde{a}_{0}(t,x,D) + \tilde{r}_{0}(t,x,D),
\eeqsn

with 
\begin{equation}
\label{lowerboundestimates}
Re\, \tilde{a}_2 \geq 0,\; Re\, (\tilde{a}_1 +t_1) + M_1 \frac{C_{a_3}}{2} \sqrt{2} \geq 0,\; Re\, \tilde{a}_0 \geq 0.
\end{equation}

Fefferman-Phong inequality applied to $Re\, \tilde{a}_2$ and sharp G{\aa}rding inequality applied to $\tilde{a}_1 + t_1+M_1 \frac{C_{a_3}}{2} \sqrt{2}$ and  $\tilde{a}_0$ give 
\beqsn
\Re\langle Re\, \tilde{a}_2(t,x,D) v,v\rangle&\geq& -c\|v\|_{L^2}^2,
\\
\Re\langle (\tilde{a}_1+t_1) (t,x,D)v,v\rangle&\geq& -\left(c+ M_1 \frac{C_{a_3}}{2} \sqrt{2}\right)\|v\|_{L^2}^2,
\\
\Re\langle \tilde{a}_0(t,x,D)v,v\rangle&\geq& -c\|v\|_{L^2}^2
\eeqsn
for a positive constant $c$. Now applying Gronwall inequality we come to the following energy estimate:
$$
\| v(t) \|^{2}_{L^2} \leq C 
\left( \| v(0) \|^{2}_{L^2} + \int_{0}^{t} \| (i\tilde P_{\Lambda} v(\tau) \|^{2}_{L^2} d\tau \right), t\in[0,T],
$$
for every $v(t,x) \in C^{1}([0,T]; \mathscr{S}(\R))$. By usual computations, this estimate provides well-posedness of the Cauchy problem associated with $\tilde P_{\Lambda}$ in $H^m(\R)$ for every $m=(m_1,m_2)\in\R^2$: for every $\tilde f\in C([0,T]; H^m(\R))$ and $\tilde g\in H^m(\R)$, there exists a unique $v \in C([0,T]; H^m(\R))$ such that $\tilde P_\Lambda v = \tilde f$, $v(0) = \tilde g$ and 
\beqs\label{eem}
\| v(t) \|^{2}_{H^m} \leq C\left( \|  \tilde g \|^{2}_{H^m} + \int_{0}^{t} \|  \tilde f(\tau) \|^{2}_{H^m} d\tau \right),\quad t\in [0,T].
\eeqs

Let us now turn back to our original Cauchy problem \eqref{equation_main_class_of_3_evolution_of_pseudo_diff_operators}, \eqref{Cauchy_problem_in_introduction}. Fixing initial data $f \in C([0,T], H^m_{\rho;s,\theta}(\R))$ and $g \in H^m_{\rho;s,\theta}(\R)$ for some $m, \rho \in \R^2$ with $\rho_2 >0$ and positive $s,\theta$ such that $\theta > s_0$, we can define $\Lambda$ as at the beginning of Section 6 with $\mu >1$ such that $s_0 > 2\mu-1$ and $M_1,M_2, k(0)$ such that \eqref{lowerboundestimates} holds. Then by Theorem \ref{theorem_continuity_infinite_order_in_gelfand_shilov_sobolev_spaces} we get
$$f_{\rho_1,\Lambda}:= e^{\rho_1 \langle D \rangle^{\frac{1}{\theta}} } e^{\Lambda }(t,x,D)f \in C([0,T], H^m_{(0,\rho_2-\bar \delta ); s, \theta}(\R))$$
and 
$$g_{\rho_1,\Lambda}:= e^{\rho_1 \langle D \rangle^{\frac{1}{\theta}} } e^{\Lambda }(0,x,D)g \in  H^m_{(0,\rho_2-\bar \delta ); s, \theta}(\R)$$
for every $\bar \delta >0$, because $1/(1-\sigma)>s$.
Since $\bar \delta$ can be taken arbitrarily small, we have that $f_{\rho_1,\Lambda} \in C([0,T], H^m)$ and $g_{\rho_1,\Lambda} \in  H^m$.
Hence the Cauchy problem
$$
\begin{cases}
 \tilde{P}_{\Lambda} v = f_{\rho_1,\Lambda} \\
v(0) = g_{\rho_1,\Lambda}
\end{cases}
$$
admits a unique solution $v \in C([0,T], H^m)\cap C^{1}([0,T], H^{m_1-3, m_2-1+1/\sigma)})$ satisfying the energy estimate \eqref{eem}.
Taking now $u= (e^{\Lambda(t,x,D)})^{-1} e^{-\rho_1\langle D \rangle^{1/\theta}}v$, 
we easily see that $u$ solves the Cauchy problem \eqref{Cauchy_problem_in_introduction}, it satisfies $$e^{\rho_1\langle D\rangle ^{1/\theta}}e^{K(T-t)\langle x\rangle_h^{1-\sigma}}e^{\tilde\Lambda(x,D)}u\in H^m(\R)$$ and it is the unique solution with this property. Namely, $
u \in C([0,T], H^m_{(\rho_1,-\tilde{\delta}); s, \theta}(\R)) \cap C^{1}([0,T], H^{(m_1-3,m_2)}_{(\rho_1 ,-\tilde{\delta}); s, \theta}(\R)) $ for every $\tilde{\delta}>0$.
Moreover, from \eqref{eem} we get
\begin{align*}\| u\|_{H_{(\rho_1,-\tilde{\delta}); s, \theta}^m} & \leq C \|v \|_{H^m} \leq C\left( \|   g_{\rho_1, \Lambda} \|^{2}_{H^m} + \int_{0}^{t} \| f_{\rho_1, \Lambda} (\tau) \|^{2}_{H^m} d\tau \right) \\
& \leq C\left( \|   g \|^{2}_{H^m_{\rho;s,\theta}} + \int_{0}^{t} \| f(\tau) \|^{2}_{H^{m}_{\rho;s,\theta}} d\tau \right).
\end{align*}
which gives \eqref{energy_estimate}
This concludes the proof. 
\qed

\begin{remark}
Notice that the argument of the proof of Theorem \ref{theorem_main_theorem}	and in particular the energy estimate  \eqref{energy_estimate} implies that the solution of the problem \eqref{Cauchy_problem_in_introduction} is unique 
in the space of all functions $u$ such that 
$$e^{\rho_1\langle D\rangle ^{1/\theta}}e^{K(T-t)\langle x\rangle_h^{1-\sigma}}e^{\tilde\Lambda(x,D)}u\in H^m (\R).$$
In general we cannot conclude that it is unique in $C([0,T], H^m_{(\rho_1, -\tilde{\delta}); s, \theta}(\R))$.
\end{remark}

\begin{remark}\label{xnelleadingterm}
In our main result we assume that the symbol of the leading term $a_3(t,D)$ is independent of $x$. This assumption is crucial in the argument of the proof. As a matter of fact, we observe that if $a_3$ depends on $x$, even allowing it to decay like $\langle x \rangle^{-m}$ for $m>>0$ , the conjugation of this term with the operator $e^{\rho_1 \langle D \rangle^{1/\theta}}$ would give
$$e^{\rho_1 \langle D \rangle^{1/\theta}} (ia_3(t,x,D)) e^{-\rho_1 \langle D \rangle^{1/\theta}} = ia_3(t,x,D_x) + \textrm{op }\left(\rho_1 \partial_\xi \langle \xi \rangle^{\frac{1}{\theta}}\cdot \partial_x a_3 \right) + \textrm{l.o.t}.$$
We observe that $$\rho_1 \partial_\xi \langle \xi \rangle^{\frac{1}{\theta}}\cdot \partial_x a_3 (t,x,\xi) \sim \langle \xi \rangle^{2+\frac1{\theta}} \langle x \rangle^{m-1}.$$
Hence since $2+\frac{1}{\theta} >2$, this remainder term could not be controlled by other lower order terms whose order in $\xi$ does not exceed $2$. Notice that this represents a difference in comparison with the $H^\infty$ frame where a dependence on $x$ in the leading term can be allowed by assuming a suitable decay assumption with respect to $x$, cf. \cite[Section 4]{ABtame}.
\end{remark}

\begin{remark}
With a major technical effort one can consider $3$-evolution equations in higher space dimension, that is for $x \in \R^n, n >1.$
At this moment, results of this type exist only for the case $p=2$, cf. \cite{ascanelli_cappiello_schrodinger_equations_Gelfand_shillov, CRJEECT, kajitani_baba_cauchy_problem_schrodinger_equations}. When passing to higher dimension, the main difficulty is the choice of the functions $\lambda_1, \lambda_2$ defining the change of variable, which must be chosen in order to satisfy certain partial differential inequalities, see also \cite[Section 4]{ABtame}. These may be non trivial for $p >2$. In this paper we prefer to restrict to the one space dimensional case both since the content is already quite technical and since the main physical models to which our results could be applied are included in this case.
\end{remark}

\noindent
\textbf{Acknowledgements.} The authors are grateful to the referees for their valuable and stimulating suggestions which helped them to improve the presentation of their results.


\end{document}